\documentclass[12pt]{article}

\usepackage[dvips]{graphics}
\usepackage[dvips]{epsfig}
\usepackage{authblk}

\newtheorem{lemma}{Lemma}

\newtheorem{corollary}[lemma]{Corollary}
\newtheorem{proposition}[lemma]{Proposition}
\newtheorem{definition}{Definition}
\newtheorem{remark}{Remark}

\newtheorem{example}{Example}
\newenvironment{proof}[1][Proof]{\textbf{#1.} }{\ \rule{0.5em}{0.5em}}

\usepackage{latexsym}
\usepackage{amsfonts}
\usepackage{amssymb}
\usepackage{amsmath}
\usepackage[english]{babel}

\def\C{\mathcal C}

\def\HH{\overline{\mathcal H}}

\def\D{\Delta}

\def\f{\varphi}
\def\G{\Gamma}
\def\g{\gamma}

\def\int{\textrm{int}}
\begin{document}
\date{}
\title{Compact $3$-manifolds via 4-colored graphs}

\bigskip

 \author[*]{Paola CRISTOFORI}
 \author[**]{Michele MULAZZANI}
\affil[*]{Dipartimento di Scienze Fisiche, Informatiche e Matematiche, Universit\`a di Modena e Reggio Emilia} \affil[**]{Dipartimento di Matematica, Universit\`a di Bologna}
\renewcommand\Authands{ and }

\maketitle

\begin{abstract}
We introduce a representation of compact 3-manifolds without spherical
boundary components via (regular) 4-colored graphs, which turns out to be very
convenient for computer aided study and tabulation. Our construction is a
direct generalization of the one given in the eighties
by S. Lins for closed 3-manifolds, which is in turn dual to the earlier construction introduced by Pezzana's school in Modena.

In this context we establish some results concerning fundamental groups,
connected sums, moves between graphs representing the same manifold,
Heegaard genus and complexity, as well as an enumeration and
classification of compact 3-manifolds representable by graphs with few vertices ($\le 6$ in the non-orientable case and $\le 8$ in the orientable one).
\end{abstract}
\bigskip

\small{

\thanks{


\smallskip

} }

\bigskip
\bigskip

\section{Introduction and preliminaries}
\label{intro}

The representation of closed 3-manifolds by 4-colored graphs has
been independently introduced by S. Lins and by Pezzana's research  group in Modena (see \cite{[FGG]} and \cite{[Li]}), by using dual constructions. A 4-colored graph is a regular edge-colored graph of valence 4, which represents a closed 3-manifold iff it satisfies certain combinatorial conditions.

The extension of the representation to
3-manifolds with boundary was performed by C. Gagliardi in \cite{[G$_0$]}
by using a slightly different class of colored graphs satisfying a notion of regularity weaker than the one required in the closed case. The study of this kind of representation has yielded several results especially with regard to the definition of combinatorial invariants and their relations with topological invariants of the represented manifolds (see \cite{[G$_4$]}, \cite{[CGG]}, \cite{[Cr$_1$]}, \cite{[CC$_2$]}). Unfortunately, Gagliardi's representation is not suitable for a satisfactory computer tabulation of non-closed 3-manifolds.

In this paper we show that any 4-colored graph, with no additional conditions, can represent a compact 3-manifold without spherical boundary components, and the whole class of such manifolds admits a representation of this type. As a consequence, an efficient computer aided tabulation of 3-manifolds with boundary can be performed by this tool.

The construction is described in Section \ref{construction} and a set of moves connecting graphs representing the same manifold is given in Section \ref{moves}. In the closed case, these moves have been proved to be sufficient to connect any two graphs representing the same manifolds (see \cite{[C$_1$]}). In the more general case of manifolds with boundary, this result is no longer true, at least when the boundary is not connected.

Examples of 4-colored graphs representing relevant classes of compact 3-ma\-ni\-folds, such as handlebodies (both orientable and non-orientable) and products of closed surfaces with the compact interval $I$, are given in Section \ref{examples}.

In Section \ref{sum} we establish the relation between the connected sum of graphs and the (possibly boundary) connected sum of the represented 3-manifolds. In Section \ref{fgroup} we associate to any 4-colored graph  a group which is strictly related to the fundamental group of the associated manifold and, therefore, it is a convenient tool for its direct computation (in many cases the two groups are in fact isomorphic).
Combinatorial invariants are defined and their relations with (generalized) Heegaard genus and Matveev complexity are discussed in Sections \ref{genus} and \ref{complexity} respectively.

The last section presents some computational results, obtained by means of
a C++ program, in terms of the number of non-isomorphic graphs
representing compact 3-manifolds with non-empty boundary, up to 12 vertices, and the number of
compact 3-manifolds admitting a graph representation up to 8 vertices in
the orientable case and up to 6 vertices in the non-orientable one.

\bigskip

In the following we fix some notations and recall some definitions and results about Heegaard splittings/diagrams which will be largely used throughout the paper.

Let  $S_g$ be a closed, connected
surface of genus $g$ either orientable (with $g\ge 0$) or non-orientable (with $g\ge 1$).
A \textit{system of curves} on $S_g$ is a (possibly empty) set of
simple closed orientation-preserving\footnote{This means that each
curve $\g_i$ has an annular regular neighborhood, as it always
happens if $S_g$ is an orientable surface.} curves $\C=\{\g_1,\ldots,\g_k\}$ on $S_g$ such that $\g_i \cap \g_j = \emptyset$, for $1\le i\neq j\le k$. Moreover, we denote by 
$V(\C)$ the set of connected components of the surface obtained by
cutting $S_g$ along the curves of $\C$. The system $\C$ is said
to be \textit{proper} if all elements of $V(\C)$ have genus zero,
and \textit{reduced} if either $\vert V(\C)\vert =1$ or no element
of $V(\C)$ has genus zero.

Note that a proper reduced system of curves on $S_g$ contains exactly $g$
curves in the orientable case and $g/2$ curves in the non-orientable one, with $g$ even. Of course, in the non-orientable case no proper system can exists when $g$ is odd. In the following $\C$ will be also considered as a subspace of $S_g$ in the obvious sense.

A \textit{compression body} $K_g$ of genus $g$ is a 3-manifold
with boundary obtained from $S_g\times I$, where $I=[0,1]$, by attaching a
finite set of 2-handles along a system of curves
(called \textit{attaching circles}) on $S_g\times\{0\}$ and
filling in with balls all the spherical boundary components of the
resulting manifold, except for $S_g\times\{1\}$ when $g=0.$ The set
$\partial_+ K_g=S_g\times\{1\}$ is called the \textit{positive}
boundary of $K_g$, while $\partial_- K_g =
\partial K_g-\partial_+ K_g$ is the \textit{negative}
boundary of $K_g$. Notice that a compression body is a handlebody
if an only if $\partial_- K_g = \emptyset$ (i.e., the system of
the attaching circles on $S_g\times\{0\}$ is proper). Obviously
homeomorphic compression bodies  can be obtained via (infinitely
many) non isotopic systems of attaching circles. Moreover, any non-reduced system of curves 
properly contains at least a reduced one inducing the same compression body. Operations of reduction correspond to elimination of complementary 2- and 3-handles.

Let $M$ be  a compact, connected 3-manifold without spherical
boundary components. A \textit{Heegaard surface} for $M$ is a
closed surface $S_g$ embedded in $M$ such that $M-S_g$ consists of
two components whose closures $K'$ and $K''$ are homeomorphic to
genus $g$ compression bodies.
The triple $(S_g, K',K'')$ is called a \textit{generalized Heegaard
splitting} of genus $g$ of $M$. It is a well-known fact that each
compact connected 3-manifold without spherical boundary components
admits a Heegaard splitting, and at least one of the two
compression bodies can be assumed to be a handlebody (in this case the splitting is simply called \textit{Heegaard splitting}).

Since two compact 3-manifolds are homeomorphic if and
only if (i) they have the same number of spherical boundary components
and (ii) they are homeomorphic after capping off by balls these
components, there is no loss of generality in
studying compact 3-manifolds without spherical boundary
components.

On the other hand, a triple $\mathcal D=(S_g, \C',\C'')$, where $\C'$ and
$\C''$ are two systems of curves on $S_g$, such that they
intersect transversally, uniquely determines a 3-manifold $M_{\mathcal D}$
corresponding to the Heegaard splitting $(S_g, K',K'')$, where
$K'$ and $K''$ are respectively the compression bodies whose
attaching circles correspond to the curves in the two systems.
Such a triple is called a \textit{generalized Heegaard diagram}
for $M_{\mathcal D}$.
In the case of a generalized Heegaard diagram $\mathcal D$ of a closed
3-manifold, both systems of curves are obviously proper; if they
are also reduced, $\mathcal D$ is simply a Heegaard diagram in the
classical sense (see \cite{[He]}).

The minimum $g$ such that a manifold $M$ admits a generalized Heegaard splitting (resp. a Heegaard splitting) of genus $g$ is called the \textit{generalized Heegaard genus} (resp. the \textit{Heegaard genus} of $M$), denoted by $\HH(M)$ (resp. by $\mathcal H(M)$). Of course the two notions coincide in the case of connected boundary.
The only 3-manifold of (generalized) Heegaard genus zero is the 3-sphere, possibly with some deleted balls. 
Examples of compact non-closed 3-manifold of generalized Heegaard genus one are $\mathbf S^1\times \mathbf S^1\times I$ and the orientable handlebody of genus one $H_1$. Of course $\HH(M)\le\mathcal H(M)$, for every manifold $M$ and it is easy to find examples where the two genera differ: for example $\HH(\mathbf S^1\times\mathbf S^1\times I)=1$ by construction but $\mathcal H(\mathbf S^1\times\mathbf S^1\times I)>1$ since the Heegaard genus of an orientable manifold can not be less than the sum of the genera of its boundary components. 

For general PL-topology and elementary notions about graphs and
embeddings, we refer to \cite{[HW]} and \cite{[Wh]} respectively.

\section{Construction}\label{construction}

Let $\G$ be a finite connected graph which is 4-regular (i.e., any vertex has valence four), possibly with multiple edges but with no loops.  A map $\g:E(\G)\to \D=\{0,1,2,3\}$ is called a \textit{$4$-coloration} of $\G$ if adjacent edges have different colors. An edge of $\G$ colored by $c\in\D$ is also called a \textit{$c$-edge}.

A \textit{$4$-colored graph} is a 4-regular graph equipped with a 4-coloration.  It is easy to see that 4-colored graphs have even order, but not all 4-regular graphs of even order can be 4-colored. Easy examples of 4-colored graphs are the graph of order two and the complete bipartite graph $K_{4,4}$. Two 4-colored graphs $\G'$ and $\G''$, with coloration $\g'$ and $\g''$ respectively, are (color-)isomorphic if there exist a graph isomorphism $\phi$ between $\G'$ and $\G''$ and a permutation $\sigma$ of $\D$ such that $\g''\circ\phi=\sigma\circ\g'$.
If $\D'\subset\D$, any connected component of the subgraph $\G_{\D'}$ of $\G$
containing exactly all $c$-edges, for each $c\in\D'$, is called a
\textit{$\D'$-residue} as well as a \textit{$\vert\D'\vert$-residue}. Of course
0-residues are vertices, 1-residues are edges and 2-residues are
bicolored cycles with an even number of edges.

We can associate a compact connected 3-manifold $M_{\G}$
without spherical boundary components  to any 4-colored graph $\G$ via the following construction.

First of all consider $\G$ as a 1-dimensional cellular complex. By
attaching to $\G$ a disk for each 2-residue we obtain a
2-dimensional polyhedron $P_{\G}$, which is special in the sense of \cite{[Ma]}. The vertices (resp. points of the edges) of $\G$ have links in $P_{\G}$ homeomorphic to a circle with three radii (resp. with two radii). Each 3-residue of $\G$, with the relative associated disks is a
closed connected surface $S$. If $S$ is a 2-sphere the residue is called \textit{ordinary} and otherwise \textit{singular}.
By attaching a 3-ball to $S$ when the residue is ordinary, and just thickening it by attaching $S\times I$ along $S\times \{0\}$ when the residue is singular, we obtain a compact connected 3-manifold $M_{\G}$ with non-spherical boundary components.
We will say that $\G$ \textit{represents} $M_{\G}$. Obviously isomorphic 4-colored graphs represent homeomorphic 3-manifolds.

For closed 3-manifolds the construction reduces to the
one introduced by Lins (see \cite{[Li]}), and it is dual to the one
introduced by Pezzana and others (see \cite{[FGG]}). Pezzana's construction was also studied by Bracho and Montejano \cite{[BM]} in the more general context of ``colored complexes'', but still within the closed case. 
So, the novelty of our construction is that it works also in case of
3-manifolds with (non-spherical) boundary.
Actually, a graph representation for manifolds
with boundary has been introduced by Gagliardi in \cite{[G$_0$]},
by using colored graphs which are not 4-regular. So,
our idea is to give, for the whole
class of compact 3-manifolds, a unitary representation  by 4-colored graphs, which seems to be more efficient than Gagliardi's for a computer aided
tabulation and classification of 3-manifolds with boundary.

\begin{remark}\label{remHeegaard} Let $i,j\in\D$ be different colors and let $\D-\{i,j\}=\{h,k\}$, then by removing from $P_{\G}$ the interior of any disk bounded either by an $\{i,j\}$-residue or by an $\{h,k\}$-residue, we obtain a closed connected surface $S_{i,j}$ which is a Heegaard surface for $M_{\G}$. Moreover, both the $\{i,j\}$-residues and the $\{h,k\}$-residues are two systems of curves $\C'$ and $\C''$ on $S_{i,j}$ such that the triple $(S_{i,j}, \C',\C'')$ is a generalized Heegaard diagram of $M_{\G}$. So, any 4-colored graph $\G$ defines three different generalized Heegaard diagrams for $M_{\G}$.
\end{remark}

\begin{figure}
\centerline{\scalebox{0.4}{\includegraphics[angle=270]{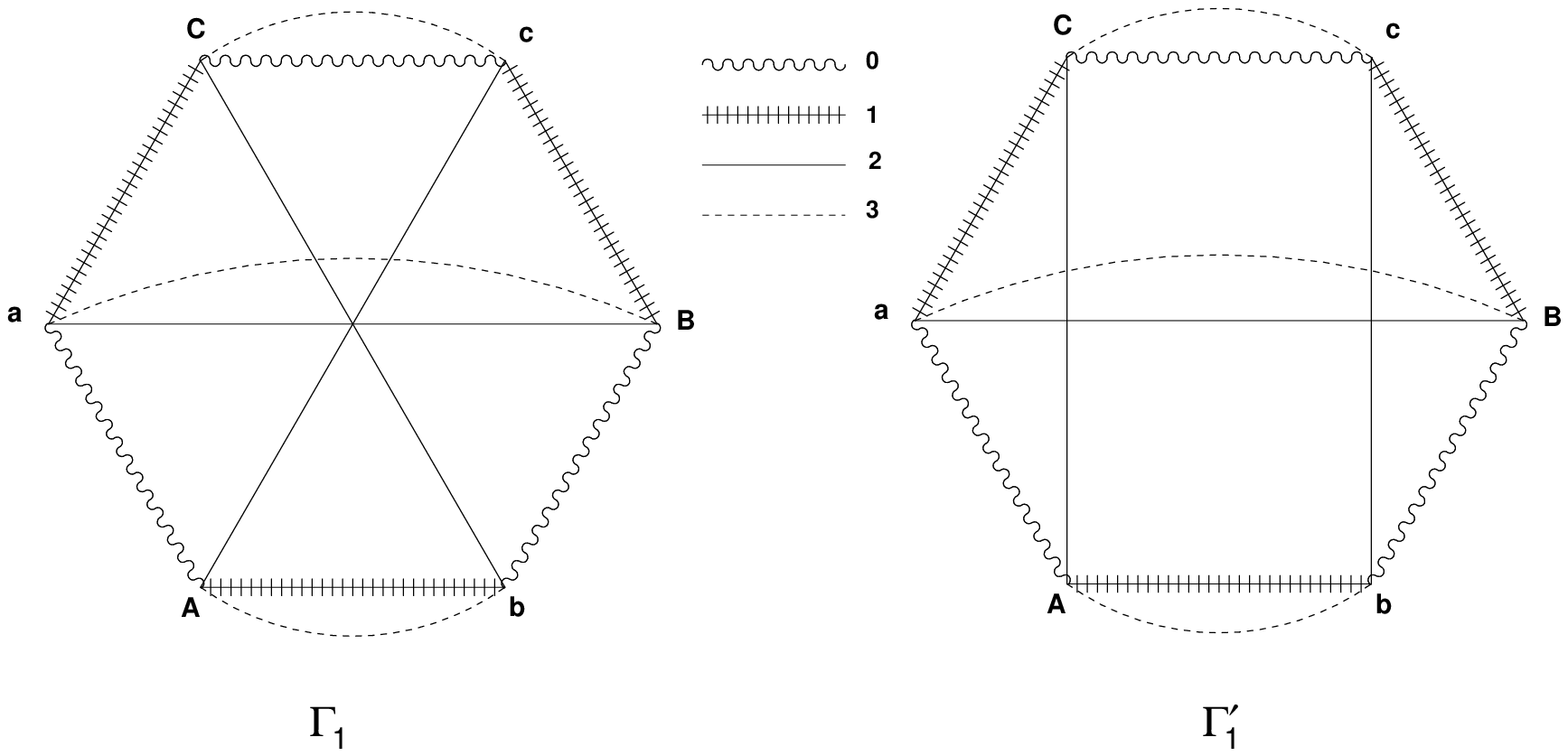}}}
\caption{The 4-colored graphs $\G_1$ and $\G'_1$, representing the genus one orientable and non-orientable handlebodies, respectively
}\label{Fig1}
\end{figure}

\begin{figure}
\centerline{\scalebox{0.5}{\includegraphics[angle=270]{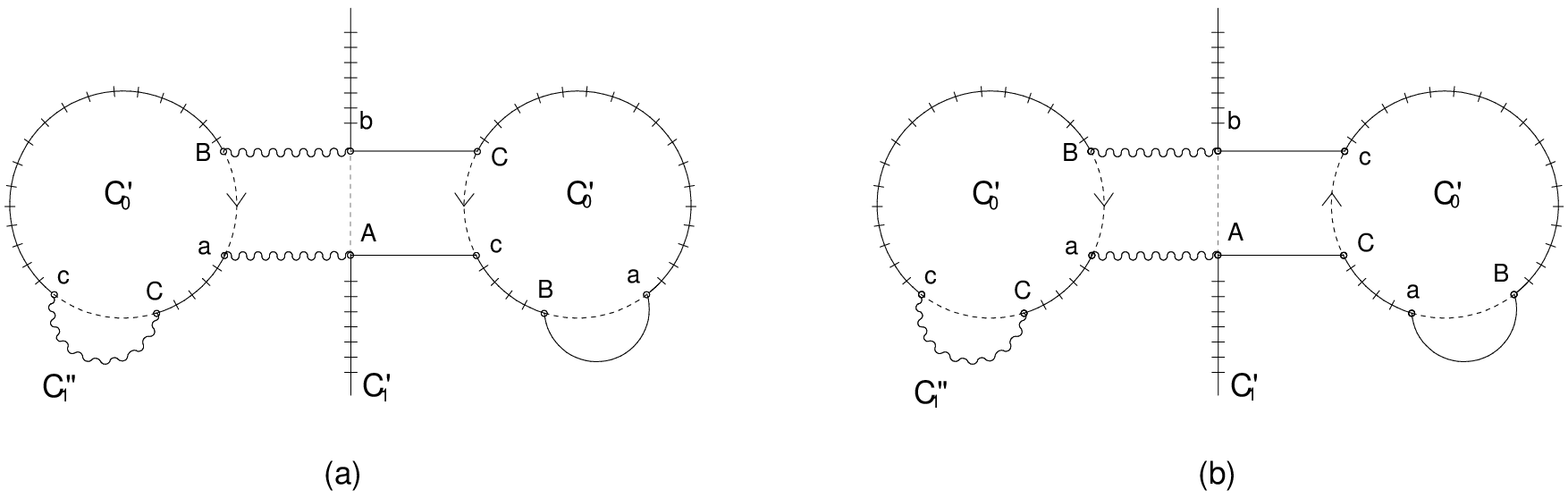}}}
\caption{(a) $\mathcal D_1$: generalized Heegaard diagram of  $M_{\G_1}$\ \ (b) $\mathcal D'_1$: generalized Heegaard diagram of $M_{\G'_1}$}\label{Fig2}
\end{figure}

\begin{example} Let us consider the 4-regular graphs $\G_1$ and $\G'_1$ of Figure \ref{Fig1}.
By Remark \ref{remHeegaard}, if we consider the pairs of colors $\{i,j\} = \{1,3\}$ and $\{h,k\} = \{0,2\}$, we obtain the generalized Heegaard diagrams $\mathcal D_1$ of $M_{\G_1}$ and $\mathcal D'_1$ of $M_{\G'_1}$ whose planar realizations are depicted in Figure \ref{Fig2}(a) and Figure \ref{Fig2}(b), respectively. Each diagram consists of the two systems of curves $\{C^\prime_0,C^\prime_1\}$ and $\{C^{\prime\prime}_1\}$ on the genus one surface $S_{1,3}$, which is orientable in the case of $\G_1$, non-orientable in the case of $\G'_1$.

Let $H_1$ be the genus one orientable (resp. non-orientable) handlebody whose 1-handles are attached along the curves $C^\prime_0$ and $C^\prime_1$ of $\mathcal D_1$ (resp. $\mathcal D'_1$). Then $M_{\G_1}$ (resp. $M_{\G'_1}$) is simply obtained from $H_1$ by adding a trivial 2-handle whose attaching curve is $C^{\prime\prime}_1$, i.e it is still homeomorphic to $H_1$.

Therefore $\G_1$ (resp. $\G'_1$) is a 4-colored graph representing the genus one orientable (resp. non-orientable) handlebody.
\end{example}

A 4-coloured graph, with 8 vertices, representing the genus one orientable handlebody was already constructed in \cite{[FG1]}.  Actually, it can also be obtained from the above graph $\G_1$, which is minimal with regard to the number of vertices, by ``adding a 2-dipole'' (see Section \ref{moves}).

The following result shows that 4-colored graphs are a representation tool for all compact 3-manifolds without spherical boundary components.

\begin{proposition}\label{representation} Any compact connected 3-manifold $M$ without spherical boundary components can be represented by a 4-colored graph.
\end{proposition}

\begin{proof} Let us consider a handle decomposition of $M$: the union of
the 0-handles and 1-handles gives a genus $g$ handlebody $H_g$,
which can be either orientable or non-orientable. Let  $D_0,\ldots,D_g$ be a system of $g+1$ pairwise disjoint disks properly
embedded in $H_g$, such that, by cutting along them, $H_g$ splits into
two balls $B'$ and $B''$. If $C'_i=\partial D_i$, for
$i=0,\ldots,g$, then $\C'=\{C'_0,\ldots,C'_g \}$ is a system of
curves on $\Sigma_g=\partial H_g$.\footnote{Note that $\Sigma_g=S_g$ in the orientable case and $\Sigma_g=S_{2g}$ in the non-orientable one.} Let $C''_j$, for $j=1,\ldots,m$, be
the attaching circles of the 2-handles. $\C''=\{C''_1,\ldots,C''_m \}$ is a system of curves on $\Sigma_g$, too.
Possibly by adding a trivial 2-handle, we can suppose that $m>0$.
Moreover, up to isotopy we can always suppose that each curve of
$\C''$ intersects transversally the curves of $\C'$; the graph
$\tilde\G=\C'\cup\C''$ is connected and cellularly embedded in $\Sigma_g$ (i.e.,
the regions of the embedding are disks). Now, let us take a set $\{N_1,\ldots ,N_m\}$ of pairwise disjoint regular
(closed) neighborhoods in $\Sigma_g$ of all curves of $\C''$, such
that $\partial N_j$ intersects transversally the curves of $\C'$ in
such a way that any component of $N_j\cap\C'$ contains exactly one
point of $\C'\cap\C''$. So, $\C=\{\partial
N_1,\ldots,\partial N_m\}$ is a system of $2m$ curves on $\Sigma_g$,
and the graph $\G=\C'\cup\C$ is a finite connected 4-regular graph
embedded on $\Sigma_g$. Let us color the arcs of $\G$ in the following
way: the arcs of the curves of $\C$ are colored by 2 if they
belong to $\partial B'$ and by 3 if they belong to $\partial B''$. Furthermore,
we color the arcs of the curves of $\C'$ by 0 if they belong to some
$N_j$ and by 1 otherwise. It is easy to see that the resulting
coloration is proper, and the 3-manifold $M_{\G}$ associated to $\G$ via
the previous construction is homeomorphic to $M$. %\qed
\end{proof}

Some properties of the manifold $M_{\G}$ correspond to properties of the representing graph $\G$. For example:

\begin{proposition}\label{bip}
$M_{\G}$ is orientable if and
only if $\G$ is bipartite.
\end{proposition}

\begin{proof}
Of course $M_{\G}$ is orientable if and only if  $S_{i,j}$ is orientable, for arbitrarily fixed $i,j\in\D$, since $M_{\G}$ is obtained from  $S_{i,j}\times I$ by adding 2-handles and possibly 3-handles.
First of all let $S_{i,j}$ be orientable, and consider the induced orientation on its 2-cells, which are disks bounded by $\{l,m\}$-residues, with $l\in\{i,j\}$ and $m\in\{h,k\}=\D-\{i,j\}$. These orientations defines for each $v\in V(\G)$ a cyclic permutation of $\D$ which is  $\sigma=(i\, h\, j\, k)$ or its inverse, corresponding to the local orientation induced on the vertices. Since permutations of adjacent vertices are inverse each other, the graph $\G$ cannot have odd cycles and therefore it is bipartite.
Viceversa, let us suppose that $\G$ is bipartite. Then $V(\G)=V'\cup V''$, where $V',V''\neq\emptyset$ and $V'\cap V''=\emptyset$, such that any $e\in E(\G)$ connects a vertex of $V'$ with a vertex of $V''$. Let $R$ be a 2-cell of $S_{i,j}$, and suppose its boundary is a  $\{l,m\}$-residue such that $m=\sigma(l)$. 
Then orient $R$ in such a way that the induced orientation on its $m$-edges goes from vertices of $V'$ to vertices of $V''$. It is easy to see that the chosen orientations on the 2-cells of $S_{i,j}$ define a global orientation for the whole $S_{i,j}$. %\qed
\end{proof}
 
Note that Propositions \ref{representation} and \ref{bip} are generalizations of analogous results obtained in \cite{[LM]}, \cite{[P]} and \cite{[BM]} for the closed case and different representation methods.

\medskip

For $i,j\in\D$, with $i\ne j$, we denote by $g_{i,j}$ the number of
$\{i,j\}$-residues of $\G$. Moreover, for each
$c\in\D$, the number of 3-residues corresponding to the colors of
$\hat c=\D-\{c\}$ will be denoted by $g_c$.
We say that $\G$ is \textit{$c$-contracted}
if either $g_c=1$ or all $\hat c$-residue are singular.
The graph is said to be \textit{contracted} if it is $c$-contracted for all $c\in\D$.
We will see in Section \ref{moves} that any 3-manifold can be represented by a contracted graph.

Moreover, boundary components of $M_{\G}$ correspond to colors
$c$ such that there exists at least a singular $\hat c$-residue. We call them \textit{singular colors}. By the
previous construction, which produces the graph $\G$ from a handle
decomposition of the manifold $M$, it is always possible to represent
a manifold by a 4-colored graph with at most one singular
color. In fact, colors 2 and 3 are not singular by construction and color 0 is non-singular since $\hat 0$-residues are attaching boundaries of the 2-handles.

A vertex of $\G$ is called a \textit{boundary vertex} if it belongs to at least one singular 3-residue, otherwise it is
called \textit{internal}. A boundary vertex is called of order $k$ if it belongs to exactly $k$ singular 3-residues. So $k\le 4$ and, as a convention, an
internal vertex is considered as a boundary vertex of order zero.
Next section shows that any 3-manifold can be represented by a graph with at least one internal vertex and, when the boundary is not empty, it can also be represented by a graph with at least a boundary vertex of order one.

\section{Moves}  \label{moves}

Given a 4-colored graph $\G$, an \textit{$h$-dipole} ($1\leq h\leq 3$) \textit{involving colors} $c_1,\ldots,c_h\in\D$ is a subgraph $\theta$ of $\G$ consisting of two vertices $v'$ and $v''$ joined by $h$ edges, colored by $c_1,\ldots,c_h$, such that $v'$ and $v''$ belong to different $\widehat{\{ c_1,\ldots, c_h\}}$-residues
of $\G$, where $\widehat{\{ c_1,\ldots, c_h\}}=\D -\{c_1,\ldots,c_h\}$. 

By \textit{cancelling} $\theta$ \textit{from} $\G$, we mean to remove $\theta$ and to paste together the hanging edges according to their colors, thus obtaining a new 4-colored graph $\G^\prime$. Conversely, $\G$ is said to be obtained from $\G^\prime$ by \textit{adding} $\theta$.
An $h$-dipole $\theta$ is called \textit{proper} if and only if $\G$ and $\G^\prime$ represent the same manifold.

\begin{proposition}\label{proper} An $h$-dipole $\theta$ of a $4$-colored graph $\G$ is proper if and only if one of the following conditions holds:
\begin{itemize} \item $h>1;$
\item $h=1$ and at least one of the $\hat c_1$-residues containing $v'$ and $v''$ is ordinary.
\end{itemize}
\end{proposition}

\begin{proof}
Let $\theta=\{v',v''\}$ be an $h$-dipole involving colors $\{c_1,\ldots,c_h\}$.

If $h=1$, let $X'$ (resp. $X''$) be the $\hat c_1$-residue of $\G$ containing $v'$ (resp. $v''$).
Cancelling $\theta$ corresponds to removing a tunnel $T$ in $M_{\G}$ connecting the two surfaces represented by $X'$ and $X''$ respectively.
The boundary of the tunnel is a cylinder $\mathbf S^1\times I$ composed by three bands $\alpha_1\times I,\alpha_2\times I,\alpha_3\times I$, whose sides $P_1\times I,P_2\times I,P_3\times I$ are portions of the $d$-edges ($d\in\hat c_1$) involved in the dipole.
It is obvious that, if the $\hat c_1$-residue containing a vertex of $\theta$, say $v'$, is ordinary, then $X'$ represents the boundary of a 3-ball and the cancellation of $\theta$ yields a new 4-colored graph still representing $M_{\G}$.

If $h=2$, then $\theta$ is a 2-component (i.e., a 2-cell) of the complex $P_\G$ which is a special spine of $M_{\G}$; then our claim follows from observing that the cancellation of $\theta$ is the inverse of the \textit{lune move} defined by Matveev and Piergallini on spines of 3-manifolds (see \cite{[M$_0$]} and \cite{[Pi]}).
Since $\theta$ is a dipole, the two $\widehat{\{c_1,c_2\}}$-residues containing $v'$ and $v''$ respectively are different 2-components of $P_\G$, so the cancellation of $\theta$ transforms $P_\G$ into another special spine of $M_\G$.

If $h=3$, then there exist two proper 1-dipoles, both involving the only color of $\widehat{\{c_1,c_2,c_3\}}$, adjacent to $v'$ and $v''$ respectively. The cancellation of $\theta$ is equivalent to the cancellation of one of these 1-dipoles. %\qed
\end{proof}

As a consequence of the above proposition, we can obtain some useful properties.

\begin{corollary} Let $M$ be a $3$-manifold without spherical boundary components, then:

\noindent (i) $M$ can be represented by a contracted $4$-colored graph;

\noindent (ii)  $M$ can be represented by a $4$-colored graph with at least an internal vertex;

\noindent (iii) $M$ can be represented by a $4$-colored graph with at least a boundary vertex of order one, if $\partial M\neq\emptyset$.
\end{corollary}

\begin{proof} Let $\G$ be  a 4-colored graph representing $M$.

(i) If $\G$ is not contracted with respect to a color $c\in\D$, then there exists a 1-dipole $\theta$, involving color $c$, such that at least one of the two $\hat c$-residues containing $v'$ and $v''$ is ordinary. Hence $\theta$ is proper and by cancelling it we obtain a new 4-colored graph which still represents $M$.
A finite sequence of such cancellations of 1-dipoles obviously yields a contracted 4-colored graph representing $M$.

(ii, iii) If $\partial M=\emptyset$ there is nothing to prove. Let $v$ be a boundary vertex with minimal order $k>0$ and let $c\in\D$ be such that the $\hat c$-residue containing $v$ is singular. By adding a $3$-dipole along the $c$-edge
containing $v$ we obtain two new vertices, $v'$ and $v''$, which
are both singular of order $k-1$. In fact, the $\hat c$-residue
containing them is obviously a 2-sphere, and for each $d\in\hat c$ any $\hat d$-residue containing them is singular if and only if the $\hat d$-residue
containing $v$ in $\G$ is singular. So by induction on $k$ we can obtain an internal vertex (resp. a boundary vertex of order one) in not more than four steps (resp. three steps). %\qed
\end{proof}

In the closed case all dipoles are proper and Casali proved in \cite{[C$_1$]} that dipole moves are sufficient to connect different 4-colored graphs representing the same manifold. A similar fact is not generally true in our context; in fact dipole moves do not change the singular colors of the involved graph, and for any 3-manifold with disconnected boundary it is easy to find two different graphs representing it, the first one with only a singular color and the second one with (at least) two singular colors.
In Section \ref{genus} another move will be introduced, the {\it bisection}, which changes the coloration but not the represented manifold. The problem whether, by adding bisection, it is possible to extend Casali's result, is currently under investigation. 

\medskip

The proof of Proposition \ref{proper}, case $h=1$, shows the effect of the cancellation of a non-proper 1-dipole.  More precisely, we have the following result:

\begin{proposition} \label{tunnel} Given a $4$-colored graph $\G$, let $t$ be a  
non-proper 1-dipole involving color $c$. If $\G'$ is the graph obtained from $\G$ by cancelling $t$, then $M_{\G'}$ is the manifold obtained from $M_{\G}$ by removing a tunnel along $t$ connecting the boundary components of $M_{\G}$ which correspond to the $\hat c$-residues involved in the dipole cancellation.
\end{proposition}

\section{Connected sums} \label{sum}

Suppose that $\G'$ and $\G''$ are two 4-colored graphs and let $v'\in V(\G')$ and $v''\in V(\G'')$. We can construct a new 4-colored graph $\G$,
called the \textit{connected sum} of $\G'$ and $\G''$ along $v'$ and $v''$, and
denoted by $\G=\G' _{v'}\#_{v''}\G''$, by removing the vertices
$v'$ and $v''$
and by welding the resulting hanging edges with the same color.

Obviously, the connected sum of two 4-colored graphs depends on the choice of the cancelled vertices. But when both vertices are internal or they are boundary vertices of the same order with respect to the same colors (the latter condition always holds, up to color permutation in one of the two graphs), then the connected sum of the graphs is strictly connected with the connected sum of the represented manifolds.

\begin{proposition} Let $\G',\G''$ be $4$-colored graphs and $v'\in V(\G'),v''\in V(\G'')$.

\noindent (i) if $v'$ and $v''$ are both internal vertices, then  $M_{\G' _{v'}\#_{v''} \G''}=M_{\G'}\#M_{\G''};$

\noindent (ii) if $v'$ and $v''$ are both boundary vertices of order one each belonging to a singular $\hat c$-residue, then  $M_{\G' _{v'}\#_{v''} \G''}=M_{\G'}\#_{\partial}M_{\G''}$, where the boundary connected sum of the manifolds is performed along the boundaries corresponding to the singular residues.
\end{proposition}

\begin{proof} (i) Since $V'$ and $V''$ are internal vertices we can perform the connected sum between $M_{\G'}$ and $M_{\G''}$ by erasing two balls $B'\subset M_{\G'}$ and $B''\subset M_{\G''}$ containing only the vertices $v'$ and $v''$ respectively and such that $L'=\partial B'\cap P_{\G'}$ and $L''=\partial B''\cap P_{\G''}$ are both homeomorphic to the 1-skeleton of a tetrahedron, where each edge belongs to a certain bicolored residue containing either $v'$ or $v''$. By gluing the two spheres $\partial B'$ and $\partial B''$ in such a way that $L'$ is glued with $L''$ coherently with the above 2-residues, we obtain a 4-colored graph $\G=\G'
_{v'}\#_{v''} \G''$ embedded in $M_{\G'}\#M_{\G''}$, and representing it in accordance with the main construction.

(ii) With regard to the case of boundary connected sum, if the singular $\hat c$-residue containing $v'$ (resp. $v''$)  is the surface $S'$ (resp. $S''$), then we can retract $S'\times I\subset M_{\G'}$ to $S'$ (resp. $S''\times I\subset M_{\G''}$ to $S''$) obtaining a new 3-manifold $M'$ homeomorphic to $M_{\G'}$ (resp. $M''$ homeomorphic to $M_{\G''}$) such that $v'\in\partial M'$ (resp. $v''\in\partial M''$). Now we can perform the boundary connected sum between $M'$ and $M''$ by erasing
two balls $B'\subset M'$ and $B''\subset M''$ containing only the vertices $v'$ and $v''$
respectively and such that $L'=\overline{\partial B'\cap \textrm{int} (M')}\cap P_{\G'}$ (resp. $L''=\overline{\partial B''\cap\textrm{int}(M'')}\cap P_{\G''}$)
is homeomorphic to the 1-skeleton of a tetrahedron, with three edges belonging to $\partial M'$ (resp. $\partial M''$)  and corresponding to a $\{i,c\}$-residue, $i\in\hat c$, containing $v'$ (resp. $v''$), and the other three edges not containing $v'$ (resp. $v''$) and corresponding to the other $\{i,j\}$-residues, $j\in\hat c$.
By gluing the two emispheres $E'=\overline{\partial B'\cap \textrm{int}(M')}$
and $E''=\overline{\partial B''\cap \textrm{int}(M'')}$
in such a way that $L'$ is glued with $L''$ coherently with the above 2-residues, we obtain a 4-colored graph $\G=\G' _{v'}\#_{v''}\G''$ embedded in $M=M'\#_{\partial}M''$. By gluing $S\times I$ to the $\hat c$-residue $S=S'\# S''$ we obtain the manifold $M_{\G}$ which is obviously homeomorphic to $M$. %\qed
\end{proof}

When the order of the boundary vertices involved in the connected sum is greater than one, more complicated topological facts occur. If $M',M''$ are two manifolds with at least two boundary components $B'_1,B'_2\subset M'$ and $B''_1,B''_2\subset M''$, we define the \textit{double boundary connected sum} of $M'$ and $M''$ as the manifold $M=M'\#_{\partial\partial} M''$ obtained by removing a tunnel from $M'$ and $M''$ connecting $B'_1$ with $B'_2$ and $B''_1$ with $B''_2$ respectively, and adding a "holed" 1-handle $\mathbf S^1\times I\times I$, in such a way that $\mathbf S^1\times I \times \{0\}$ is attached to $B'_1$ and $\mathbf S^1\times I \times \{1\}$ is attached to $B''_1$ as in Figure \ref{Fig3}. An example of such operation is the  double boundary connected sum of $S_{g'}\times I$ with $S_{g''}\times I$, performed by choosing $\{P'\}\times I$ and $\{P''\}\times I$ respectively as tunnels, for any $P'\in S_{g'}$ and $P''\in S_{g''}$. It is easy to see that, if both $S_{g'}$ and $S_{g''}$ are orientable, the sum is homeomorphic to $S_{g'+g''}\times I$.

Of course, the double boundary connected sum in general depends on the choice of tunnels, but in the next proposition tunnels are trivial, and the sum is uniquely defined, up to homeomorphism.

\begin{figure}
\centerline{\scalebox{0.3}{\includegraphics[angle=270]{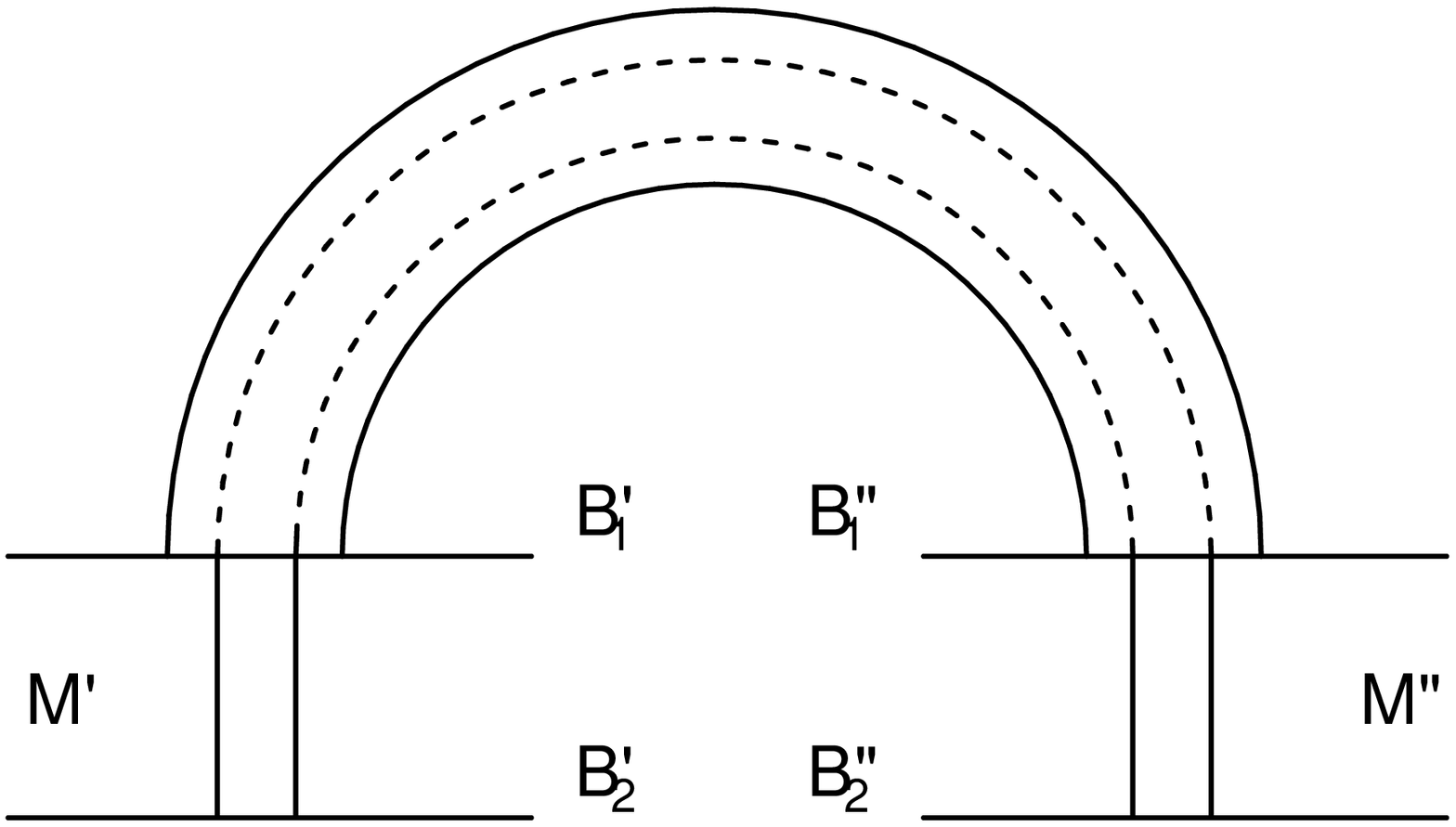}}}
\caption{Double boundary connected sum}\label{Fig3}
\end{figure}

\begin{proposition} Let $\G',\G''$ be $4$-colored graphs and $v'\in V(\G'),v''\in V(\G'')$.
If $v'$ and $v''$ are both boundary vertices of order two such that they belong to a singular $\hat c$-residue and to a singular $\hat d$-residue, with $c\neq d$, then
$M_{\G' _{v'}\#_{v''} \G''}=M_{\G'}\#_{\partial\partial} M_{\G''}$, where the double boundary connected sum of the manifolds is performed between the boundaries corresponding to the singular residues.
\end{proposition}

\begin{proof} Without loss of generality we can suppose $c=0$ and $d=1$. Add a $3$-dipole to the $0$-edge containing $v'$, as well as to the $0$-edge containing $v''$, obtaining two new 4-colored graphs called again $\G'$ and $\G''$.
The new vertices $w',u'\in V(\G')$ and $w'',u''\in V(\G'')$ are boundary vertices of order one.
Performing $\G' _{w'}\#_{w''} \G''$ we obtain a graph $\G$ representing
$M=M_{\G'}\#_{\partial}M_{\G''}$, where the boundary connected sum is made by connecting by a 1-handle $Y$ the boundaries corresponding to the $\hat 1$-residues containing
$v'$ and $v''$ respectively. After that we remove the $3$-dipole
containing $u'$ and $u''$, obtaining a new graph $\G_1$ still representing $M$.
Now the vertices $v'$ and $v''$ are boundary vertices of order two (with respect to colors 0 and 1), and they are connected by a $0$-edge, which is a non-proper 1-dipole, and which is the core of the 1-handle $Y$.
The result is achieved by applying Lemma \ref{tunnel}, using as tunnel a regular neighborhood of $Y_1\cup Y_2$, where $Y_1=\{v'\}\times I\subset S_1\times I$,
$Y_2=\{v''\}\times I\subset S_2\times I$, where $S_1$ and $S_2$ are respectively the $\hat 0$-residue and $\hat 1$-residue containing $v'$ (see Figure \ref{Fig3}). %\qed
\end{proof}

\begin{remark}Note that the graph $\G_1$ appearing in the proof of the above proposition can be simply obtained directly from $\G'$ and $\G''$ by ''switching'' the two 0-colored edges containing $v'$ and $v''$. More precisely, if we call $\bar v'$ (resp. $\bar v''$) the vertex of $\G'$ (resp. $\G''$) 0-adjacent to $v'$ (resp. $v''$), we join $v'$ with $v''$ and $\bar v'$ with $\bar v''$ by a 0-colored edge.
More generally, we can obtain a graph representing the boundary connected sum performed along the boundary components corresponding to two given $\hat c$-residues $X'\subset\G'$ and $X''\subset\G''$, by switching two $d$-colored edges $e'$ and $e''$ belonging to $X'$ and $X''$ respectively and such that for each $i\notin\{c,d\}$, $e'$ and $e''$ belong to ordinary $\hat i$-residues.
\end{remark}

\section{Basic examples}\label{examples}

Any color of a 4-colored graph $\G$ can be interpreted as a fixed point free involution on $V(\G)$. When $\G$ is bipartite with vertex bipartition $V'$ and $V''$, a color $c\in\D$ can also be interpreted as a bijection $f_c:V'\to V''$, and the maps $\f_1=f_1^{-1}\circ f_0$, $\f_2=f_2^{-1}\circ f_0$ and $\f_3=f_3^{-1}\circ f_0$ are permutations of $V'$ which completely determine $\G$, up to isomorphism.

\subsection{4-colored graphs representing $S_g\times I$}\label{example1}

Let $S_g$ be a closed orientable (resp. non-orientable) surface of
genus $g$ and let $\G$ be the 4-colored bipartite (resp. non-bipartite) graph with $p=2(2g+1)$ (resp. $p=2(g+1)$) vertices obtained from the standard 3-colored graph representing $S_g$ described in \cite{[G$_5$]},
by adding 3-edges parallel (i.e. having the same endpoints) to the 0-edges (see Figures \ref{Fig4} and \ref{Fig5}).

\begin{figure}
\centerline{\scalebox{0.48}{\includegraphics[angle=270]{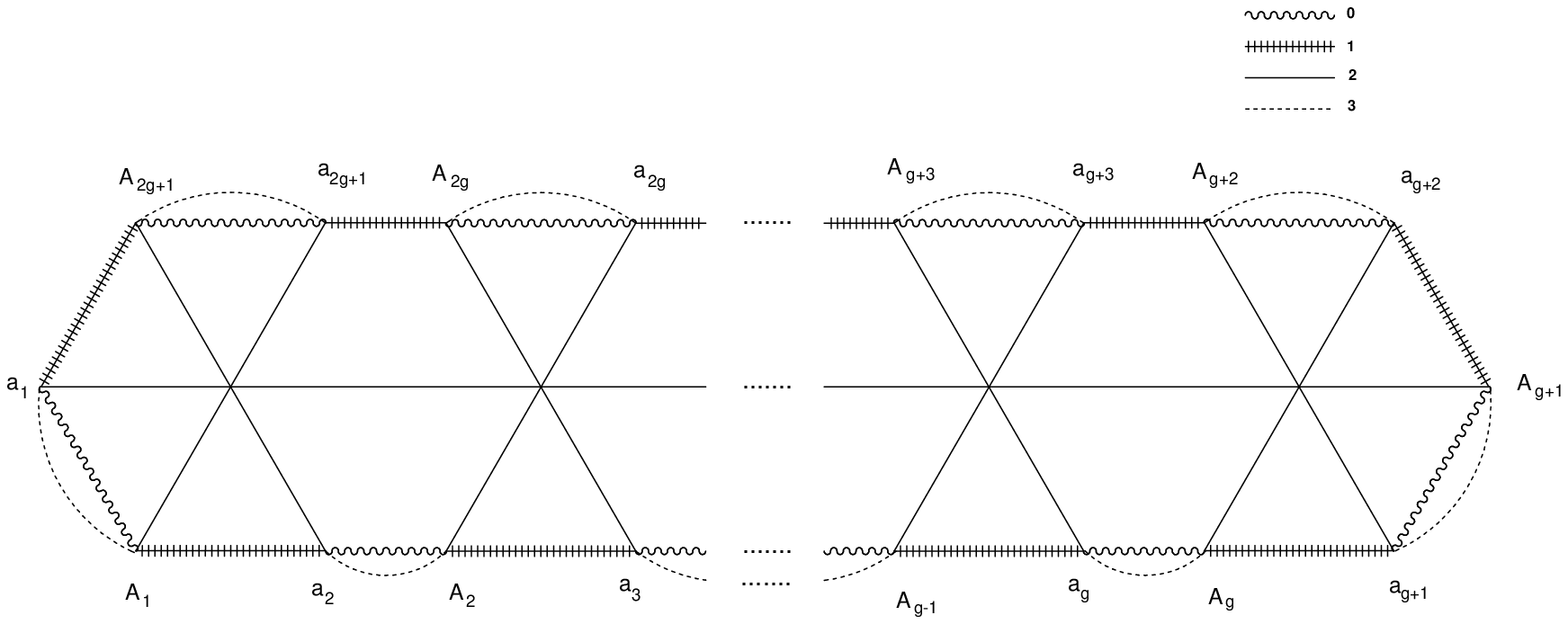}}}
\caption{4-colored graph representing $S_g\times I$, orientable case}\label{Fig4}
\end{figure}

\begin{figure}
\centerline{\scalebox{0.48}{\includegraphics[angle=270]{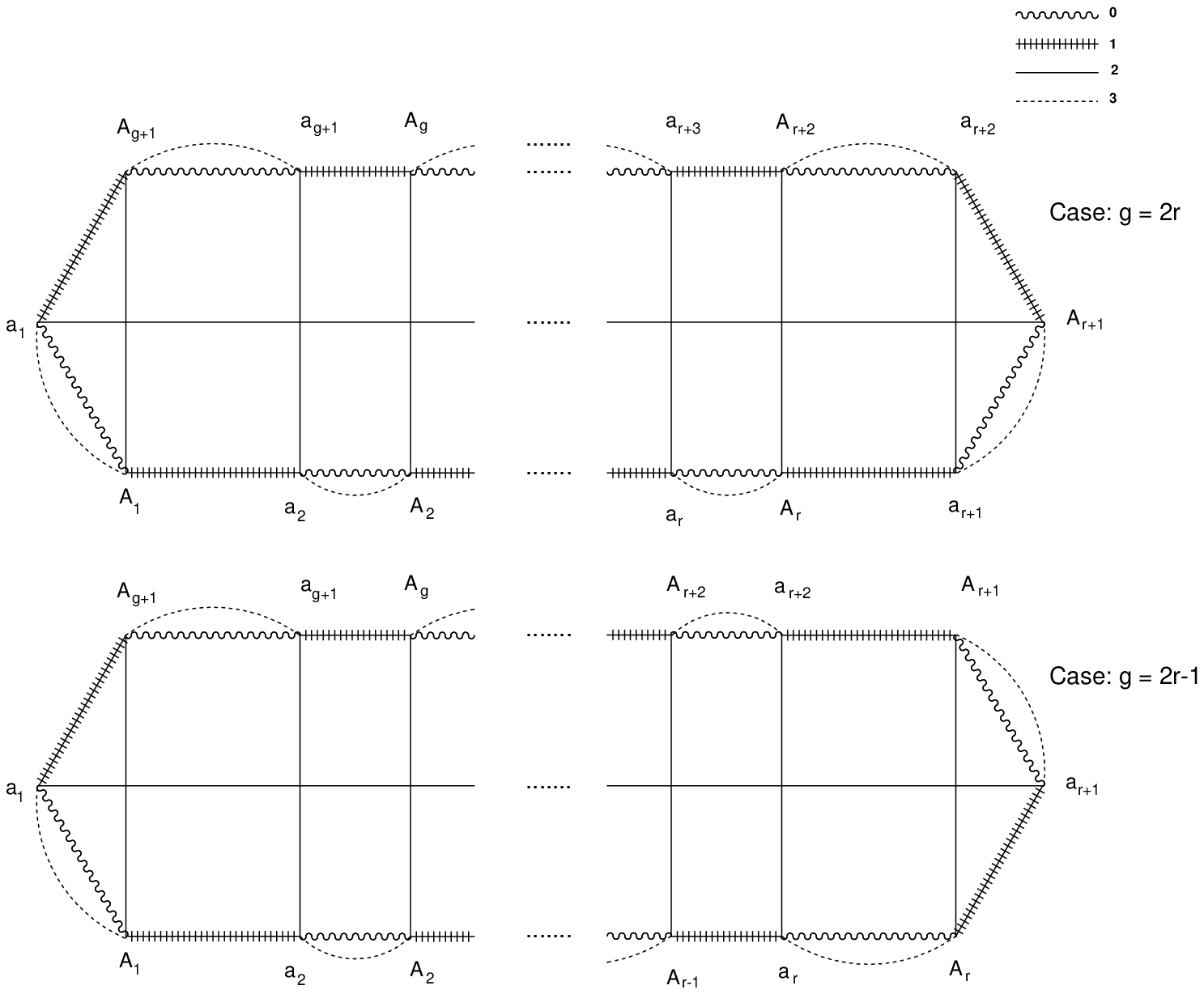}}}
\caption{4-colored graph representing $S_g\times I$, non-orientable case}\label{Fig5}
\end{figure}

Note that the singular colors of $\G$ are $0$ and $3$, while $1$ and $2$ are not singular. Moreover, $\G$ is contracted.

Let us consider the generalized Heegaard diagram for $M_\G$ associated to $\G$ and the pair $\{0,2\}$. The Euler characteristic of the Heegaard surface $S_{0,2}$ can be computed via the cellular decomposition induced by $\G$ on it. More precisely, we have:
$$\chi(S_{0,2})=p-2p+g_{0,1}+g_{1,2}+g_{2,3}+g_{0,3}=-p+1+1+1+\frac p2=3-\frac p2\ .$$
Hence, $S_{0,2}$ has genus $g$ and therefore it is homeomorphic to $S_g$.

The system of curves $\C^\prime$ (resp. $\C^{\prime\prime}$) on $S_{0,2}$
consisting of the $\{0,2\}$-residues (resp. $\{1,3\}$-residues) contains a
single curve $l$ (resp. $l^\prime$), and it is neither proper nor reduced.
In fact, $l$ (resp. $l^\prime$) bounds a disk on $S_{0,2}$; therefore it
can be removed from $\C^\prime$ (resp. $\C^{\prime\prime}$) without
changing the related compression body, which is homeomorphic to $S_g\times I$.
As a consequence, we obtain a reduced generalized Heegaard diagram for
$M_\G$, where both systems of curves on $S_{0,2}$ are empty. Hence $\G$
represents $S_g\times I$.

Note that the above 4-colored graph for the case $g>1$ can be also obtained from the one re\-pre\-senting $S_1\times I$  by performing iterated connected sums of this graph with itself, which correspond to double boundary connected sums as described in the previous section.

\subsection{4-colored graphs representing handlebodies}\label{example2}

A representation of the genus $g$ handlebody $H_g$ can be easily obtained starting from the one of $H_1$ given in Section \ref{construction} and then performing boundary connected sums.

In the orientable case, we start with the solid torus $H_1$, which can be represented by the bipartite 4-colored graph $\G_1$ depicted in Figure \ref{Fig1}: it has 6 vertices $a,b,c,A,B,C$ and its coloring can be described by $f_0(x)=X$, and the three permutations of $\{a,b,c\}$: $\f_1=(a\,b\,c)$, $\f_2=(a\,c\,b)$ and $\f_3=(a\,b)$. 
All vertices of $\G_1$ are boundary vertices of order one, corresponding to a singular $\hat 3$-residue, which is a torus.

By performing the connected sum of two copies of $\G_1$ along any pair of vertices, we obtain a bipartite graph $\G_2$ representing $H_2$ with 10 vertices $a,b,c,d,e,A,B,C,$ $D,E$ and defined by the permutations $\f_1=(a\,b\,c\,d\,e)$, $\f_2=(a\,e\,b\,d\,c)$ and $\f_3=(a\,b\,c)$ (see Figure \ref{Fig6}).

\begin{figure}
\centerline{\scalebox{0.3}{\includegraphics[angle=270]{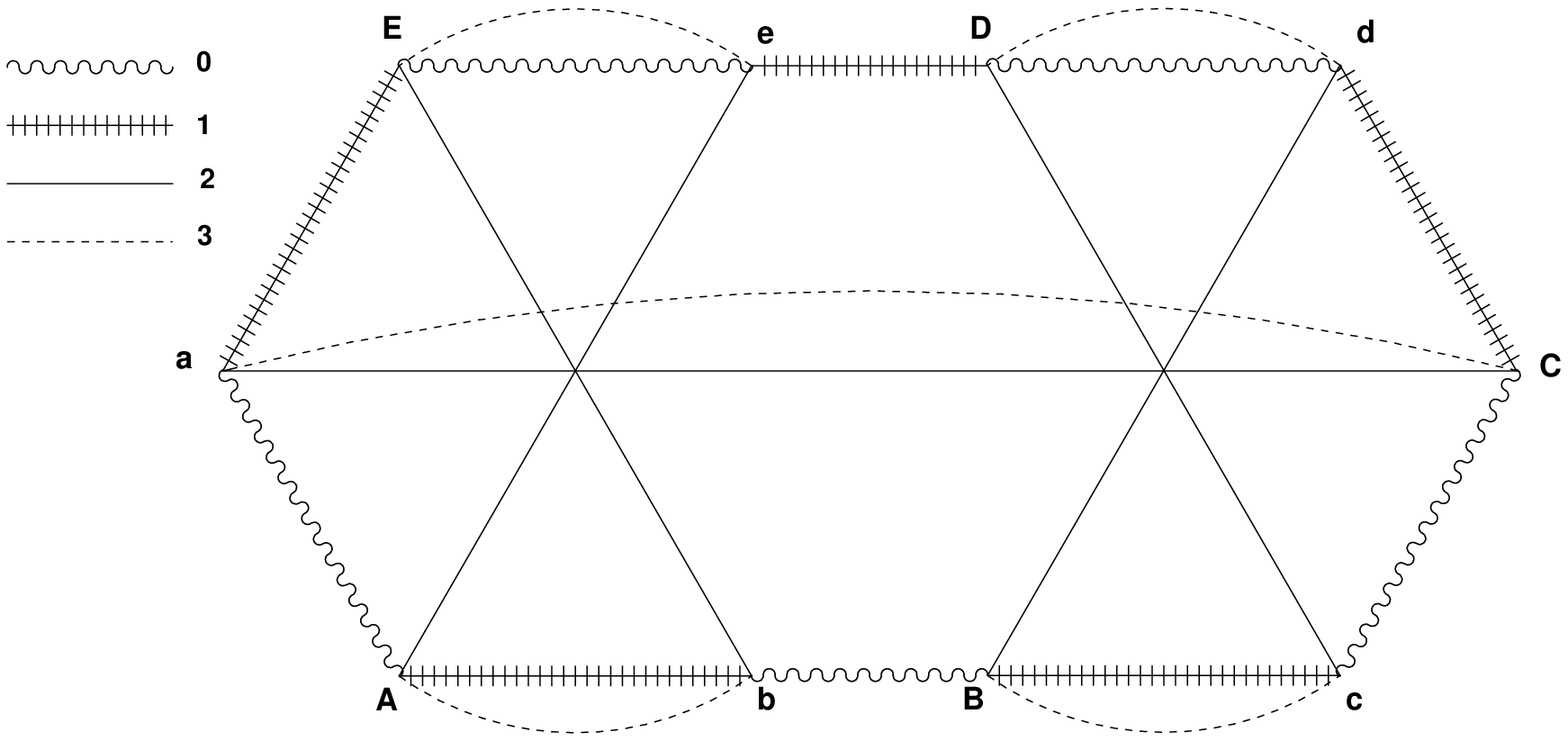}}}
\caption{4-colored graph representing $H_2$, orientable case}\label{Fig6}
\end{figure}

In order to get a 4-colored graph representing $H_g$, we iterate the boundary connected sum, taking care, for each $i=2,\ldots,g-1$, to perform it with respect to the vertex $a$ of $\G_1$ and the ``rightmost'' vertex of $\G_i$. As a consequence we obtain (see Figure \ref{Fig7}):

\begin{figure}
\centerline{\scalebox{0.5}{\includegraphics[angle=270]{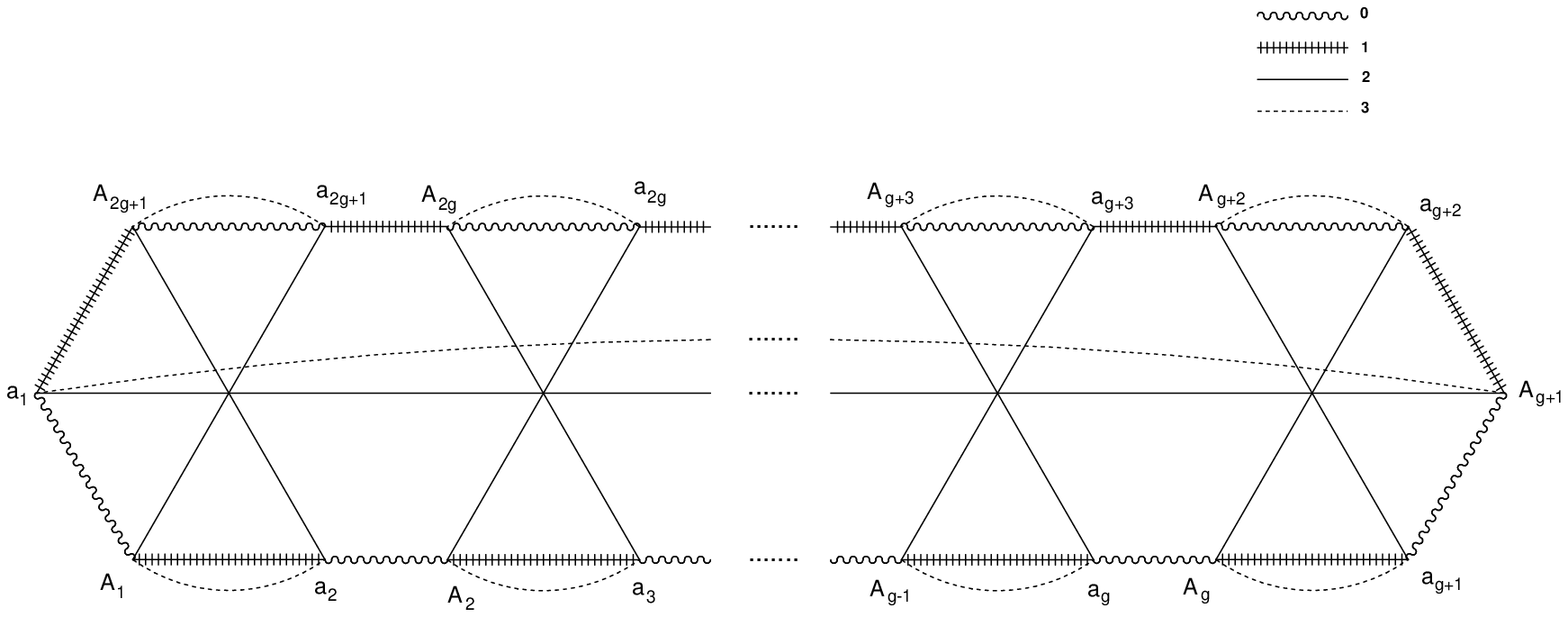}}}
\caption{4-colored graph representing $H_g$, orientable case}\label{Fig7}
\end{figure}

\begin{proposition} \label{orienthandle}
The genus $g$ orientable handlebody $H_g$ is represented by a bipartite $4$-colored graph $\G_g$ with $4g+2$ vertices $a_1,\ldots,a_{2g+1},A_1,\ldots,A_{2g+1}$, defined by the permutations
\begin{align*}
\f_1&=(a_1\,a_2\,\ldots\,a_{2g+1}),\, \f_2=(a_1\,a_{2g+1}\,a_2\,a_{2g}\,\ldots\,a_i\,a_{2g+2-i}\,\ldots\,a_g\,a_{g+2}\,a_{g+1}),
\\ \f_3&=(a_1\,a_2\,\ldots\,a_{g+1}).
\end{align*}
\end{proposition}

In the non-orientable case, we start with the solid Klein bottle $H_1$, which can be represented by the (non-bipartite) 4-colored graph $\G'_1$ of Figure \ref{Fig1}: it has 6 vertices $a,b,c,A,B,C$, the same $c$-edges ($c\in\{0,1,3\}$) as $\G_1$, and the 2-edges connecting $b$ with $c$, $A$ with $C$ and $a$ with $B$.
All vertices are boundary vertices of order one, corresponding to a singular $\hat 3$-residue which is a Klein bottle.

By performing the connected sum of two copies of $\G'_1$ along any pair of vertices, we obtain a graph $\G'_2$ representing $H_2$ with 10 vertices $a,b,c,d,e,A,B,C,D,E$ with the same $c$-edges ($c\in\{0,1,3\}$) as $\G_2$ and 2-edges connecting $A$ with $E$, $b$ with $e$, $B$ with $D$, $c$ with $d$ and $a$ with $C$ (see Figure \ref{Fig8}).

\begin{figure}
\centerline{\scalebox{0.35}{\includegraphics[angle=270]{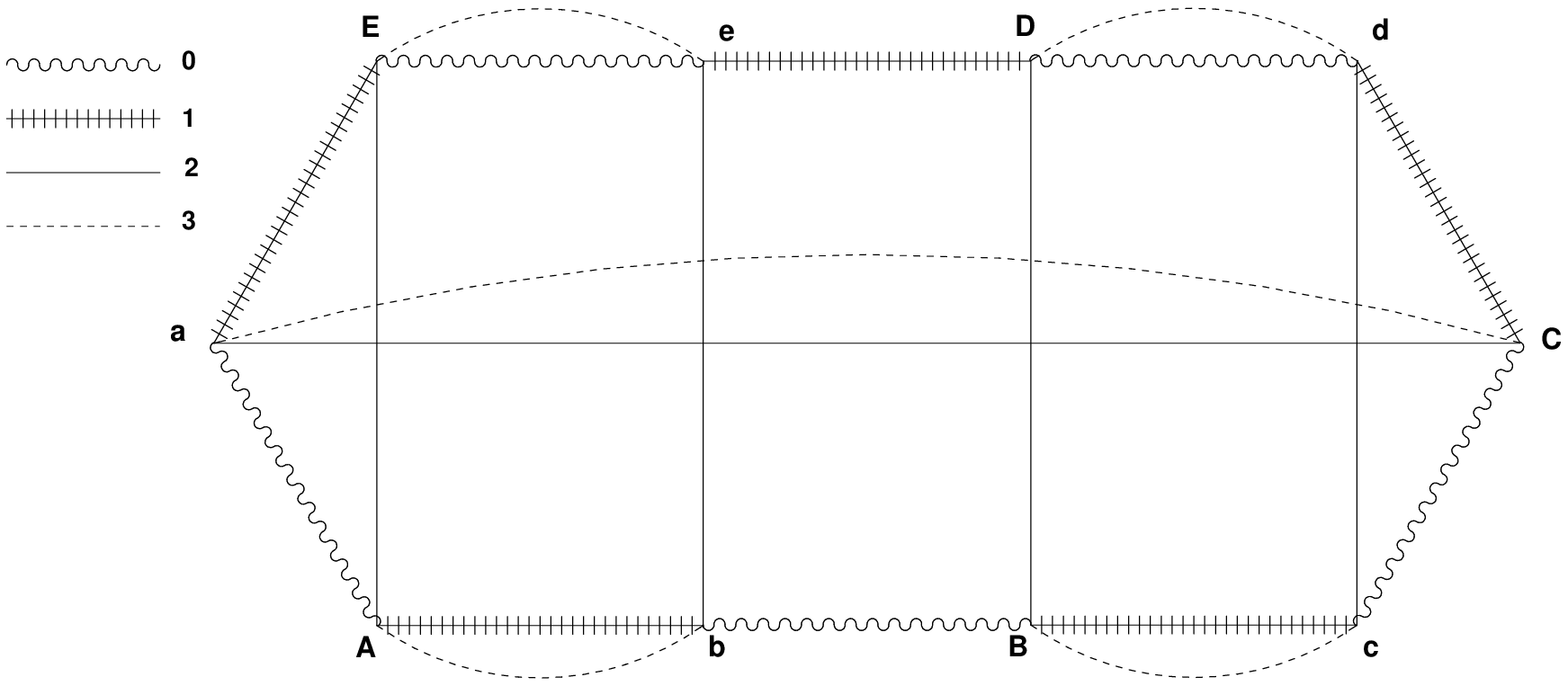}}}
\caption{4-colored graph representing $H_2$, non-orientable case}\label{Fig8}
\end{figure}

As in the orientable case, for each $i=2,\ldots,g-1$, we perform the boundary connected sum with respect to the vertex $a$ of $\G'_1$ and the ``rightmost'' vertex of $\G'_i$, thus obtaining (see Figure \ref{Fig9}):

\begin{figure}
\centerline{\scalebox{0.48}{\includegraphics[angle=270]{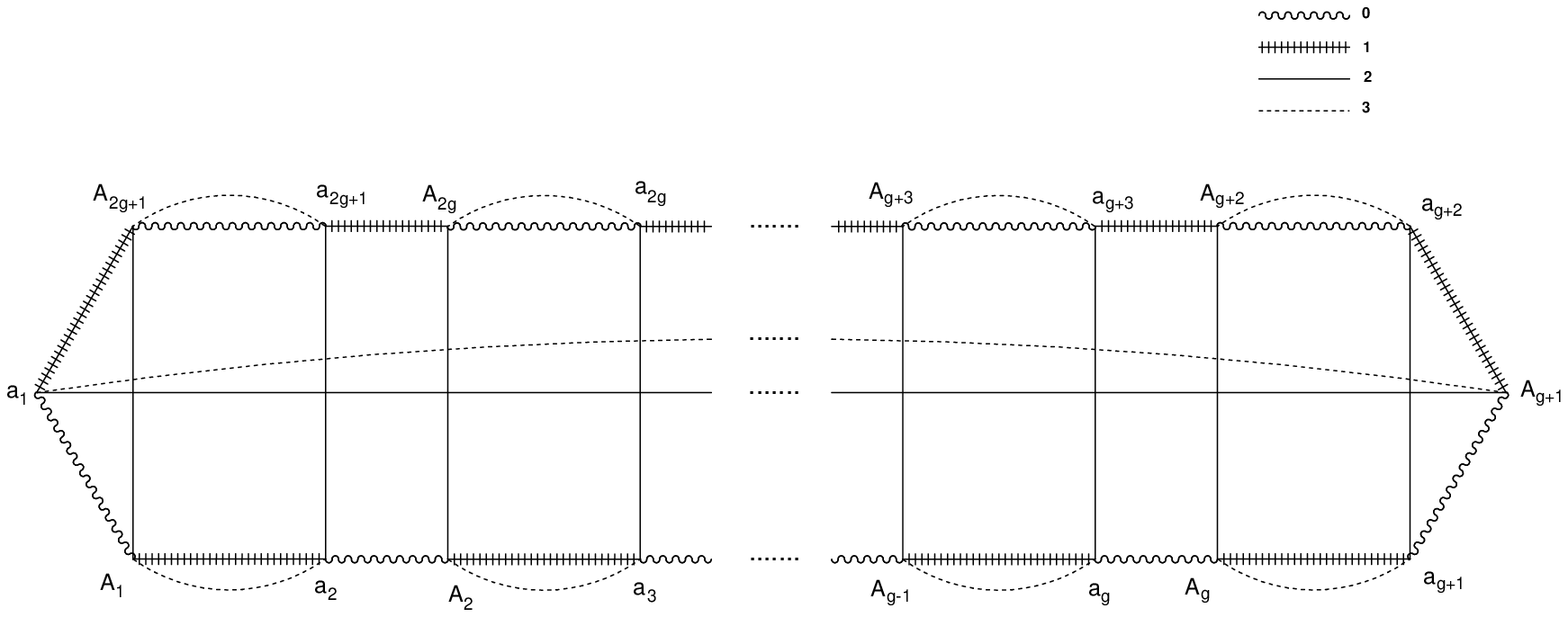}}}
\caption{4-colored graph representing $H_g$, non-orientable case}\label{Fig9}
\end{figure}

\begin{proposition}  \label{nonorienthandle}
The genus $g$ non-orientable handlebody $H_g$ is represented by a $4$-colored graph $\G'_g$ with $4g+2$ vertices $a_1,\ldots,a_{2g+1},A_1,\ldots,A_{2g+1}$,
with the same $c$-edges, $c\in\{0,1,3\}$, as the graph $\G_g$ of Proposition $\ref{orienthandle}$ and $2$-edges
connecting $a_1$ with $A_{g+1}$, $a_i$ with $a_{2g+3-i}$, for $i=2,\ldots,g+1$, and $A_i$ with $A_{2g+2-i}$, for $i=1,\ldots,g$.
\end{proposition}

\section{Fundamental group}\label{fgroup}

If $\G$ is a 4-colored graph, then the fundamental group of the
represented manifold $M_{\G}$ coincides with the fundamental group of the
associated 2-dimensional polyhedron $P_{\G}$, since $M_{\G}$ is obtained from $P_{\G}$ by adding  to $P_{\G}$ 3-balls and pieces which are retractable to $P_\G$.
Therefore, the computation of  $\pi_1(M_{\G})$ is a routine algebraic topology exercise: a finite presentation has generators corresponding to edges which are not in a fixed spanning tree of $\G$ and relators corresponding to all 2-residues of $\G$.

In several cases the group can be obtained by selecting a particular class of edges and 2-residues, as follows.
Let $c\in\D$ be any color, we define the \textit{$c$-group} of $\G$ as
the group $\pi(\G,c)$ generated by all $c$-edges (with a fixed arbitrary
orientation) and whose relators correspond to all $\{i,c\}$-residues, for each
$i\in \hat c$. Just give an orientation to any involved 2-residue,
choose a starting vertex and follow the cycle according
to orientation. The relator is obtained by taking the $c$-edges of
the cycle in the order they are reached in the path and with the
exponent $+1$ or $-1$ according to whether the orientation of the
edge is coherent or not with the one of the cycle.

In general $\pi(\G,c)$ depends on $c$, but when $c$ is a
non-singular color, the group is strictly connected with the fundamental group
of $M_{\G}$ (see \cite{[LM]} and \cite{[Gr]} for the case of closed manifolds).

\begin{proposition} \label{group} Let $\G$ be a $4$-colored graph, and $c$ be a non-singular
color for $\G$. Then $\pi_1(M_{\G})$ is a quotient of $\pi(\G,c)$,
obtained by adding to the relators a minimal set of $c$-edges which connect
$\G_{\hat c}$.
\end{proposition}

\begin{proof}
The group $\pi_1(M_{\G})=\pi_1(P_{\G})$ is isomorphic to the fundamental group of the space $X$ obtained by adding to $P_{\G}$ only the 3-balls corresponding to the $\hat c$-residues. The space $X$ has the same homotopy type of a 2-complex with 0-cells corresponding to the $\hat c$-residues, 1-cells corresponding to the $c$-edges of $\G$ and 2-cells corresponding to the $\{c,i\}$-residues of $\G$, for $i\in\hat c$. So the result is straightforward. %\qed
\end{proof}

\begin{corollary} Let $\G$ be a $4$-colored graph, and $c$ be a non-singular
color for $\G$ such that $g_c=1$,
then $\pi_1(M_{\G})\cong \pi(\G,c).$
\end{corollary}

\section{Generalized regular genus}\label{genus}

A cellular embedding of a 4-colored graph $\G$ into a closed
surface $S$ is called \textit{regular} if there exists a cyclic permutation
$\varepsilon=(\varepsilon_0\,\varepsilon_1\,\varepsilon_2\,\varepsilon_3)$
of $\D$ such that any region of the embedding is bounded by a
$\{\varepsilon_i,\varepsilon_{i+1}\}$-residue of $\G$, for $i\in\mathbb
Z_4$.

We recall the following result from \cite{[G$_3$]}:

\begin{proposition}\label{reg-emb}
Let $\G$ be a bipartite (resp. non-bipartite) $4$-colored graph. Then:

\noindent (i) for any cyclic permutation $\varepsilon$ of $\D$, the graph $\G$ regularly embeds into a closed orientable (resp. non-orientable) surface $S_\varepsilon$ of Euler characteristic $$\chi(S_\varepsilon)=\sum_{i\in\mathbb Z_4}
g_{\varepsilon_i,\varepsilon_{i+1}}-p\ ,$$
where $p$ is the number of vertices of $\G$;

\noindent (ii) up to equivalence there exist exactly three regular embeddings of $\G$ into closed orientable (resp. non-orientable) surfaces, one for each cyclic permutation of $\D$, up to inversion, and there exist no regular embeddings of $\G$ into non-orientable (resp. orientable) surfaces.

\end{proposition}

\medskip

We denote by $\rho(\G)$ the minimum genus of $S_\varepsilon$ among all
cyclic permutations $\varepsilon$ of $\D$.

\begin{definition} \emph{Given  a compact 3-manifold $M$, the }
generalized regular genus \emph{of $M$ is:}

$$\overline{\mathcal G}(M)=\min\ \{ \rho(\G)\ |\ \G \text{\emph{
represents }} M\}.$$
\end{definition}

\medskip

As observed in Remark \ref{remHeegaard}, any 4-colored graph
defines three generalized Heegaard splittings of the represented 3-manifold,
which are induced by the choice of two colors of $\D$. Since
any two colors $i,j\in\D$ define (up to inversion) a cyclic
permutation of $\D$ where they are non-consecutive,
the following relation between regular embeddings of $\G$ and generalized Heegaard
splittings of $M_\G$ can be established.

\begin{proposition}\label{reg-H-spl}
Given a $4$-colored graph $\G$, for each unordered pair
$\{\varepsilon,\varepsilon^{-1}\}$, where $\varepsilon=(\varepsilon_0\,\varepsilon_1\,\varepsilon_2\,\varepsilon_3)$ is a cyclic
permutation of $\D$, there exists a generalized Heegaard splitting
$(S_\varepsilon,K'_\varepsilon,K''_\varepsilon)$ of $M_\G$, such that $\G$
regularly embeds into $S_\varepsilon$ and $K'_\varepsilon$ (resp.
$K''_\varepsilon$) is obtained from $S_\varepsilon\times I$ by attaching $2$-handles
on $S_\varepsilon\times\{0\}$ (resp. on $S_\varepsilon\times\{1\}$) along the
$\{\varepsilon_0,\varepsilon_2\}-$residues (resp.
$\{\varepsilon_1,\varepsilon_3\}-$residues).
Moreover,  $(S_\varepsilon,K'_\varepsilon,K''_\varepsilon)$ is a Heegaard
splitting (i.e. at least one of $K'_\varepsilon$ or $K''_\varepsilon$ is a
handlebody) if and only if there are two non-singular colors which are non-consecutive in $\varepsilon$.
\end{proposition}

As a consequence we have:

\begin{corollary} If $M$ is any compact $3$-manifold, then
$\HH(M) \le \overline{\mathcal G}(M)$.
\end{corollary}

Let us denote by $\mathcal G(M)$ the minimum $\rho(\G)$ where $\G$
is taken among all 4-colored graphs representing $M$ and having at most
one singular color.
It is proved in \cite{[CGG]} and \cite{[Cr$_1$]} that $\mathcal G(M)$
coincides with the \textit{regular genus} of $M$, as originally defined by
Gagliardi\footnote{Gagliardi's definition of regular genus of a 3-manifold
with non-empty boundary was given in \cite{[G$_4$]} through a
representation of the manifold by means of 4-\textit{colored graphs regular with respect to color 3}
(i.e., obtained from a 4-colored graph by deleting
some 3-edges). In fact, an analog of Proposition \ref{reg-emb}
holds for these graphs, by a suitable
adaptation of the concept of regular embedding into surfaces with
boundary.}, and that the regular genus equals the Heegaard genus (resp. is twice the
Heegaard genus) of an orientable (resp. non-orientable) 3-manifold.

Obviously $\overline{\mathcal G}(M)\leq\mathcal G(M)$ and there exist
3-manifolds for which the strict inequality holds, as
proved in the following proposition.

\begin{proposition} Let $S_g$ be a closed surface of genus $g$, then

$$\overline{\mathcal G}(S_g\times I)=g<\mathcal G(S_g\times I)=2g.$$
\end{proposition}

\noindent\begin{proof} The regular genus of $S_g\times I$ has been proved
to be $2g$ in \cite{[BR]}.
In order to prove the first equality, let $\G$ be the graph described in
Section \ref{example1} as representing $S_g\times I$ (see Figure \ref{Fig4} or
\ref{Fig5}, according to the orientability of $S_g$).
The genus $g$ Heegaard surface $S_{0,2}$ described in Section
\ref{example1}, is precisely the surface $S_\varepsilon$ (with
$\varepsilon = (0\ 1\ 2\ 3)$) into which $\G$ regularly embeds.
Hence $\overline{\mathcal G}(S_g\times I)\leq g$.
On the other hand, by Remark \ref{gen_bordo} below, we have
$\overline{\mathcal G}(S_g\times I)\geq g$.
This completes the proof. %\qed
\end{proof}

However, if $M$ has connected boundary, then it follows easily from the
construction in Section~\ref{construction} that  $\overline{\mathcal
G}(M)=\mathcal G(M)$. More generally, we can state:

\begin{proposition} For each compact 3-manifold $M$, we have

$$\overline{\mathcal G}(M)=\min\ \{ \rho(\G)\ |\ \G \text{\emph{
represents }} M\text{\ \emph{and has at most two singular colors}}\}.$$
\end{proposition}

\medskip

The result is a direct consequence of Lemma \ref{gen_2col} below.
In order to prove it, we first introduce a suitable
transformation on 4-colored graphs (see \cite{[G$_1$]}).

Given $a,b\in\D$, let $X$ be a $\hat a$-residue, with
$V(X)=\{v_1,\ldots,v_r\}$, and $b\in \hat a$. A \textit{bisection of type} $(a,b)$ on $X$ is a transformation of $\G$ producing a new 4-colored graph
$\bar\G$ in the following way (see Figure \ref{Fig10} for a local
picture around a vertex of $X$):
\begin{itemize}
\item [-] add to $V(\G)$ a set $\bar V(X)=\{\bar v_1,\ldots,\bar v_r\}$
in bijective correspondence with $V(X)$;
\item [-] for each $i\in\D-\{a,b\}$, add a $i$-colored edge between two
vertices of $\bar V(X)$ if and only if the corresponding vertices of $V(X)$ are
$i$-adjacent;
\item [-] substitute all $b$-colored edges of $X$ by $a$-colored edges
with the same endpoints;
\item [-] for each $j=1,\ldots,r$, add a $b$-colored edge between $v_j$
and $\bar v_j$.
\end{itemize}

\begin{figure}
\centerline{\scalebox{0.3}{\includegraphics[angle=270]{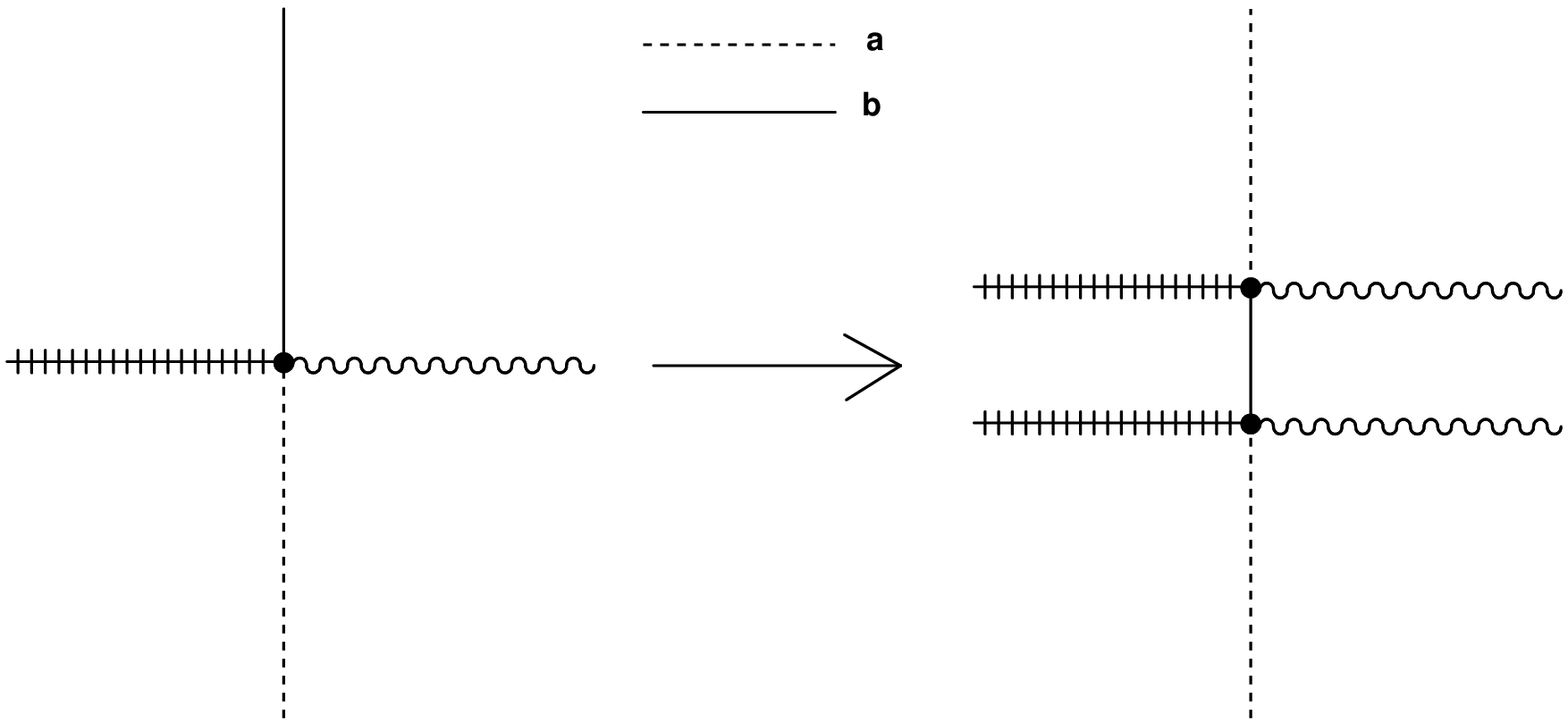}}}
\caption{A bisection of type $(a,b)$}\label{Fig10}
\end{figure}

\begin{lemma}\label{gen_2col}
Let $\G$ be a $4$-colored graph representing a compact $3$-manifold $M$, then
there
exists a $4$-colored graph $\G^\prime$ representing $M$, having at most two
singular colors and such that $\rho(\G)=\rho(\G^\prime)$.
\end{lemma}

\begin{proof}
Suppose that $\G$ has more than two singular colors, and let $\varepsilon$ be a cyclic
permutation of $\D$ such that the surface $S_\varepsilon$ where $\G$
regularly embeds has genus $\rho(\G)$.
Let $X$ be a singular $\hat a$-residue of $\G$ and $\bar\G$
the graph obtained by performing a bisection of type
$(a,b)$ on $X$, where $b\in\hat a$ is a color which is not consecutive to $a$ in $\varepsilon$.
It is easy to see that $\bar\G$ has one more singular
$\hat b$-residue (which is isomorphic to $X$) and one less
singular $\hat a$-residue than $\G$, while the number of
singular residues of the other two colors does not change.
An easy computation shows that the embedding surface of $\bar\G$ with
respect to $\varepsilon$ has still genus $\rho(\G)$.

On the other hand, if we compare the generalized Heegaard splittings
associated to $\varepsilon$ of $\G$ and $\bar\G$ respectively, we see that
the systems of curves defined by the
$\{a,b\}$-residues are the same, while the system
of curves defined by the $\{c,d\}$-residues in
$\bar\G$, where $\{c,d\}=\D-\{a,b\}$, is obtained by simply doubling the corresponding system for $\G$.
Hence $\bar\G$ still represents $M$.
By performing bisections of type $(a,b)$ on all
singular  $\hat a$-residues of $\G$, we obtain a 4-colored
graph $\G^{\prime\prime}$ representing $M$, having the same genus as $\G$
and such that color $a$ is not singular.
Finally, by performing bisections of type $(c,d)$
on all singular  $\hat c$-residues of $\G^{\prime\prime}$, the
required 4-colored graph $\G^\prime$ is obtained. %\qed
\end{proof}

\begin{remark}\label{gen_bordo}Given a compact 3-manifold $M$ with $h$
boundary components $B_1,\ldots,B_h$, let $g_i$ be the
genus  of $B_i$ if it is orientable and half of its genus otherwise.
It is known from \cite{[G$_4$]} that $\mathcal G(M)\geq\sum_{i=1}^h g_i$ (resp. $\mathcal
G(M)\geq 2\cdot\sum_{i=1}^h g_i$) if $M$ is orientable (resp. non-orientable).
Example \ref{example1} shows that this inequality does not hold for
the generalized regular genus. However, by Proposition \ref{reg-H-spl}, we
can state
\begin{align*}\overline{\mathcal G}(M)\geq\min\{\max\{\sum_{i\in I^\prime}
g_i,\sum_{j\in I^{\prime\prime}} g_j\}\ |&\ I^\prime,\
I^{\prime\prime}\subseteq\{1,\ldots,h\},\ I^\prime\cup
I^{\prime\prime}=\{1,\ldots,h\},\\ &\ I^\prime\cap
I^{\prime\prime}=\emptyset\}
\end{align*}
(resp.
\begin{align*}
\overline{\mathcal G}(M)\geq
2\cdot\min\{\max\{\sum_{i\in I^\prime} g_i,\sum_{j\in I^{\prime\prime}}
g_j\}\ |&\ I^\prime,\ I^{\prime\prime}\subseteq\{1,\ldots,h\},\
I^\prime\cup I^{\prime\prime}=\{1,\ldots,h\},\\ &\ I^\prime\cap
I^{\prime\prime}=\emptyset\}\ ).
\end{align*}

\end{remark}

\section{Complexity}\label{complexity}

The \textit{(Matveev) complexity} $c(M)$ of a compact 3-manifold $M$ was defined in \cite{[M$_1$]} as the minimal number of true vertices of an almost simple spine of $M$. In this section we will define a concept of complexity starting from generalized Heegaard splittings/diagrams of compact 3-manifolds, which will turn out to be an upper bound for the value of Matveev complexity.

Let us consider a non-reduced system of curves $\C$ on a surface $S_g$, and let $G(\C)$ be the graph which is dual to the one
determined by $\C$ on $S_g$ (i.e., the vertex-set of $G(\C)$ is in one-to-one correspondence with $V(\C)$ and the edges correspond to
curves of $\C$). Denote by $V^+(\C)$ the set of
vertices of $G(\C)$ corresponding to the components with positive genus.
Now let $T$ be a subgraph of $G(\C)$ satisfying the following conditions:

\begin{itemize}
\item [(i)] $T$ contains
all vertices of $G(\C)$,
\item  [(ii)] if $V^+(\C)=\emptyset$ then
$T$ is a tree of $G(\C)$,
\item  [(iii)] if $V^+(\C)\ne \emptyset$ then
each connected component of $T$ is a tree
containing exactly one vertex of $V^+(\C)$;
\end{itemize}
and let $\mathcal A(\C)$ be the set of all such subgraphs of $G(\C)$.
Observe that any choice of an element $T \in \mathcal A(\C)$ yields a reduced system on $S_g$ obtained by removing from $\C$  the curves corresponding to the edges of $T$ (i.e., removing complementary 2- and 3-handles from the complex).

Let  $\mathcal D=(S_g, \C ',\C '')$ be a generalized Heegaard diagram of a compact 3-manifold $M$ without spherical boundary components.
If at least one of the two systems of curves is non-reduced, we can associate to $\mathcal D$ several reduced diagrams, still representing $M$, obtained by reducing both systems of curves.
For a reduced generalized diagram $\mathcal{\tilde D}=(S_g, \mathcal{\tilde C}',\mathcal{\tilde C}'')$, we call \textit{singular vertices} of $\mathcal{\tilde D}$ the points of $\mathcal{\tilde C}'\cap\mathcal{\tilde C}''$ and denote their number by $n(\mathcal{\tilde D})$. Moreover, we call a \textit{region} of  $\mathcal{\tilde D}$ any connected component $\mathcal R$ of $S_g-(\mathcal{\tilde C}'\cup\mathcal{\tilde C}'')$ and denote by $n(\mathcal R)$ the number of singular vertices belonging to its boundary.

The \textit{modified Heegaard complexity} $\tilde c(\mathcal{\tilde D})$ of $\mathcal{\tilde D}$ is:

$$\tilde c(\mathcal{\tilde D})=\begin{cases}n(\mathcal{\tilde D})-\max\{n(\mathcal R)\ |\ \mathcal R\text{\ is a region of }\mathcal{\tilde D}\}\quad&\text{if either\ }\mathcal{\tilde C}'\text{\ or\ }\mathcal{\tilde C}''\text{\ is proper}\\ n(\mathcal{\tilde D})\quad&\text{otherwise}\end{cases}.$$

Then we define the \textit{modified Heegaard complexity} of a compact 3-manifold $M$ as:

$$\tilde c(M) = \min\,\{\tilde c(\mathcal{\tilde D}) \mid \mathcal{\tilde D}  \text{\ is a generalized\ reduced\ Heegaard \ diagram \  of \ } M\}.$$

In \cite{[CMV]}  modified Heegaard complexity was originally defined by means of (non-generalized) Heegaard splittings and in that setting it was proved to be an upper bound for the value of Matveev complexity.

We can now generalize that result.

\begin{proposition} For each compact 3-manifold without spherical boundary components $M$, we have
$$\tilde c(M)\leq c(M).$$
\end{proposition}

\begin{proof}
Let $\mathcal{\tilde D}=(S_g, \mathcal{\tilde C}',\mathcal{\tilde C}'')$ be a reduced generalized Heegaard diagram for $M$ such that $\tilde c(\mathcal{\tilde D})=\tilde c(M)$.
If $\mathcal{\tilde D}$ is a Heegaard diagram, i.e. one of the two systems of curves is proper, the statement was already proved in \cite[Proposition 2.3]{[CMV]}.
Suppose that neither $\mathcal{\tilde C}'$ nor $\mathcal{\tilde C}''$ are proper. Let $P$ be the 2-complex obtained by attaching to  $S_g$ the 2-cells which are the cores of the 2-handles corresponding to $\mathcal{\tilde C}'$ and $\mathcal{\tilde C}''$.
Then $M-P$ collapses to $\partial M$, i.e. $P$ is a spine of $M$. Moreover, $P$ is special and its true vertices are exactly the singular vertices of $\mathcal{\tilde D}$. Hence, $\tilde c(\mathcal{\tilde D})\leq c(M)$. %\qed
\end{proof}

We recall that any 4-colored graph $\G$ representing a compact 3-manifold defines three generalized (non-reduced) Heegaard diagrams.
Therefore, by reducing the diagrams as described above, we can compute their modified Heegaard complexity and thus get upper bounds for Matveev complexity of the represented manifold.

Observe that, for any closed 3-manifold $M$, the complexity $\tilde c(M)$ is always realized by a Heegaard diagram associated to a 4-colored graph representing $M$, as proved in  \cite{[CCM]} (see also \cite{[C$_2$]} and \cite{[CC$_1$]}).

\section{Computational results}\label{computation}

The combinatorial nature of colored graphs makes them particularly suitable for computer manipulation. In particular, it is possible to generate catalogues of 4-colored graphs for increasing number of vertices, in order to analyze the represented compact 3-manifolds. In this context, we can obtain interesting manifolds even with a low number of vertices.

By the results of Section \ref{moves} , without loss of generality we can restrict the catalogues to contracted graphs with no 2-dipoles. Therefore, given a positive integer $p$, the catalogue of contracted 4-colored graphs with $2p$ vertices and no 2-dipoles is generated algorithmically in the following way. First of all the set $\mathcal S^{(2p)}$ of all (possibly disconnected) 3-colored graphs $\G$ with $2p$ vertices, such that either $\G$ is connected or each connected component of $\G$ represents a surface of positive genus is constructed. After that, the 3-edges are added to each graph of $\mathcal S^{(2p)}$ in all possible ways which give rise to a contract 4-colored graph without 2-dipoles. Moreover, since the Euler characteristic of the represented manifold can be easily computed directly through the graph, we can drop closed manifolds, since they have been already catalogued (see \cite{[BCrG$_1$]}, \cite{[CC]} and also \cite{[BCCGM]}).

\medskip

A 4-colored graph $\G$ is called $m$-bipartite if all $r$-residues are bipartite, for any $r\leq m\leq 4$, and there exists a non-bipartite $(m+1)$-residue when $m<4$. Obviously any $\G$ is at least $2$-bipartite. Note that, by construction and by Proposition \ref{bip}, $\G$ is $4$-bipartite (resp. $3$-bipartite) if and only if $M_\G$ is orientable (resp. non-orientable with orientable boundary).
The \textit{code} of a $m$-bipartite 4-colored graph $\G$ with $2p$ vertices is a ``string'' of length  $(7-m)p$ which
completely describes both combinatorial structure and coloration of $\G$ (see also
\cite{[CG]}). The importance of the code as a tool for representing 4-colored graphs relies on the following result:   

\begin{proposition} {\rm (\cite{[CG]})} Two $4$-colored graphs are color-isomorphic
if and only if they have the same code.\end{proposition}

As a consequence, by representing each colored graph by its code, we can
produce catalogues containing only graphs which are pairwise not color-isomorphic.

The following table shows the output data, up to 12 vertices, of a C++
program implementing the generating algorithm.

$C^{(2p)}$ (resp. $\tilde{C}^{(2p)}$) denotes the catalogue of bipartite (resp. 2-bipartite) contracted 4-colored graphs with $2p$ vertices and no 2-dipoles.

$C^{(2p)}_c$ (resp. $\tilde{C}^{(2p)}_c$) denotes the subset of $C^{(2p)}$ (resp. $\tilde{C}^{(2p)}$) consisting of the graphs with connected boundary and $C^{(2p)}_t$ (resp. $\tilde{C}^{(2p)}_t$) denotes the subset of $C^{(2p)}_c$ (resp. $\tilde{C}^{(2p)}_c$) consisting of those graphs having toric boundary. Each row of the table shows the cardinality of the corresponding set.

\begin{table}
\caption{\textbf{Catalogues up to 12 vertices.}}\label{table1}
 \centerline{ \begin{tabular}{|c||c|c|c|c|c|c|}
  \hline
   {\bf 2p } & 2 & 4 & 6 & 8 & 10 & 12\\
 \hline\hline \ & \ & \ & \ & \ & \ & \ \\
  {\bf $C^{(2p)}$} & 0 & 0 & 2 & 4 & 57 & 902\\ 
\hline \ & \ & \ & \ & \ & \ & \ \\
 {\bf $\tilde{C}^{(2p)}$} & 0 & 1 & 6 & 90 & 3967 & 395877\\ 
 \hline \ & \ & \ & \ & \ & \ & \ \\
  {\bf $C^{(2p)}_t/C^{(2p)}_c$} & 0/0 & 0/0 & 1/1 & 0/0 & 0/5 & 26/28\\
\hline \ & \ & \ & \ & \ & \ & \ \\
  {\bf $\tilde{C}^{(2p)}_t/\tilde{C}^{(2p)}_c$} & 0/0 & 0/0 & 0/1 & 0/0 & 0/10 & 24/73\\
   \hline
\end{tabular}}
\end{table}

With regard to the classification of the manifolds represented by the above graphs, we state some results concerning cases with few number of vertices.

\begin{proposition} There exist exactly seven non-closed compact non-orientable $3$-manifolds without spherical boundary components, which can be represented by a $4$-colored graph of order $\le 6$.
\end{proposition}

\begin{proof}
The unique element of $\tilde{C}^{(4)}$ is the graph already described in Section \ref{examples} representing $\mathbf {RP}^2\times I$.
The six elements of $\tilde{C}^{(6)}$ have all distinct boundaries, which are also different from $\mathbf {RP}^2\cup\mathbf {RP}^2\cong \partial(\mathbf {RP}^2\times I)$.

More precisely, one of the graphs in $\tilde{C}^{(6)}$ represents the genus one non-orientable handlebody (see Section \ref{example2}) and the other five have the following boundaries:
\begin{itemize}
\item [-] four Klein bottles;
\item [-] two Klein bottles;
\item [-] one Klein bottle and two projective planes;
\item [-] one torus, one Klein bottle and two projective planes;
\item [-] four projective planes.
\end{itemize} %\qed
\end{proof}

\begin{proposition}\label{class_or} There exist exactly five non-closed compact orientable $3$-manifolds without spherical boundary components, which can be represented by a $4$-colored graph of order $\le 8$.
\end{proposition}

\begin{proof}
The two elements of $C^{(6)}$ are the graphs already described in Section \ref{examples} representing $\mathbf S^1\times\mathbf S^1\times I$ and the genus one orientable handlebody $H_1$ respectively.

We list below the four elements of $C^{(8)}$ by means of their codes. The vertex-set is $\{a,b,c,d,A,B,$ $C,D\}$ and the coloration is defined by $f_0(x)=X$ and $f_i(x)=Y$ $(0< i\leq 3$), where $x$ is the $j$-th letter of the alphabet and $Y$ is the $(4(i-1)+j)$-th letter in the string of the code $(1\leq j\leq 4)$.
\begin{eqnarray*}
\G_1 : DABCCDABBCDA\\
\G_2 : DABCDCABCADB\\
\G_3 : DABCDCABCBDA\\
\G_4 : DABCDCABCDBA
\end{eqnarray*}

The graphs $\G_1$ and $\G_3$ (resp. $\G_2$ and $\G_4$) have both boundary consisting of four (resp. three) tori.

The fundamental groups of $M_{\G_1}$ and $M_{\G_3}$ admit the following presentations: 
\begin{align*}
\pi_1(M_{\G_1})&=<x_1,x_2,x_3,x_4\mid x_1x_2x_1^{-1}x_2^{-1},\  x_3x_4x_3^{-1}x_4^{-1},\  x_3x_2x_4x_1x_3^{-1}x_2^{-1}x_4^{-1}x_1^{-1}>,\\
\pi_1(M_{\G_3})&=<x_1,x_2,x_3,x_4\mid x_1x_2x_1^{-1}x_2^{-1},\  x_3x_4x_3^{-1}x_4^{-1},\ x_2x_4x_2^{-1}x_4^{-1}>.
\end{align*}

Since the two groups have different number of subgroups with index $\le 6$, as checked by GAP program \cite{[GAP2013]}, the manifolds represented by $\G_1$ and $\G_3$  are distinct.

With regard to the manifolds represented by $\G_2$ and $\G_4$ (see Figure \ref{Fig11}), by choosing color $0$, which is the only non-singular color of these graphs, 
it is easy to see that both fundamental groups admit the presentation 
$$<\ x,y,z\mid xz = zx,\ yz=zy\ >,$$
and hence they are isomorphic to $(\mathbb Z *\mathbb Z)\times \mathbb Z$.

Since the above group is not a free product, the manifolds $M_{\G_2}$ and $M_{\G_4}$ are irreducible. 
A compact, orientable, irreducible 3-manifold with toroidal boundary whose fundamental group is a (non-trivial) direct product is homeomorphic to a trivial bundle over $\mathbf S^1$, where the fiber is a compact surface (see \cite{[E]}, \cite{[HeJ]} and also \cite{[He]}). Therefore, in our case, we have $M_{\G_2}\cong M_{\G_4}\cong \mathbf S^1\times F$, where $F$ is $\mathbf S^2$ with the interior of three disjoint disks deleted.  %\qed
\end{proof}

\medskip

\begin{figure}
\centerline{\scalebox{0.4}{\includegraphics[angle=270]{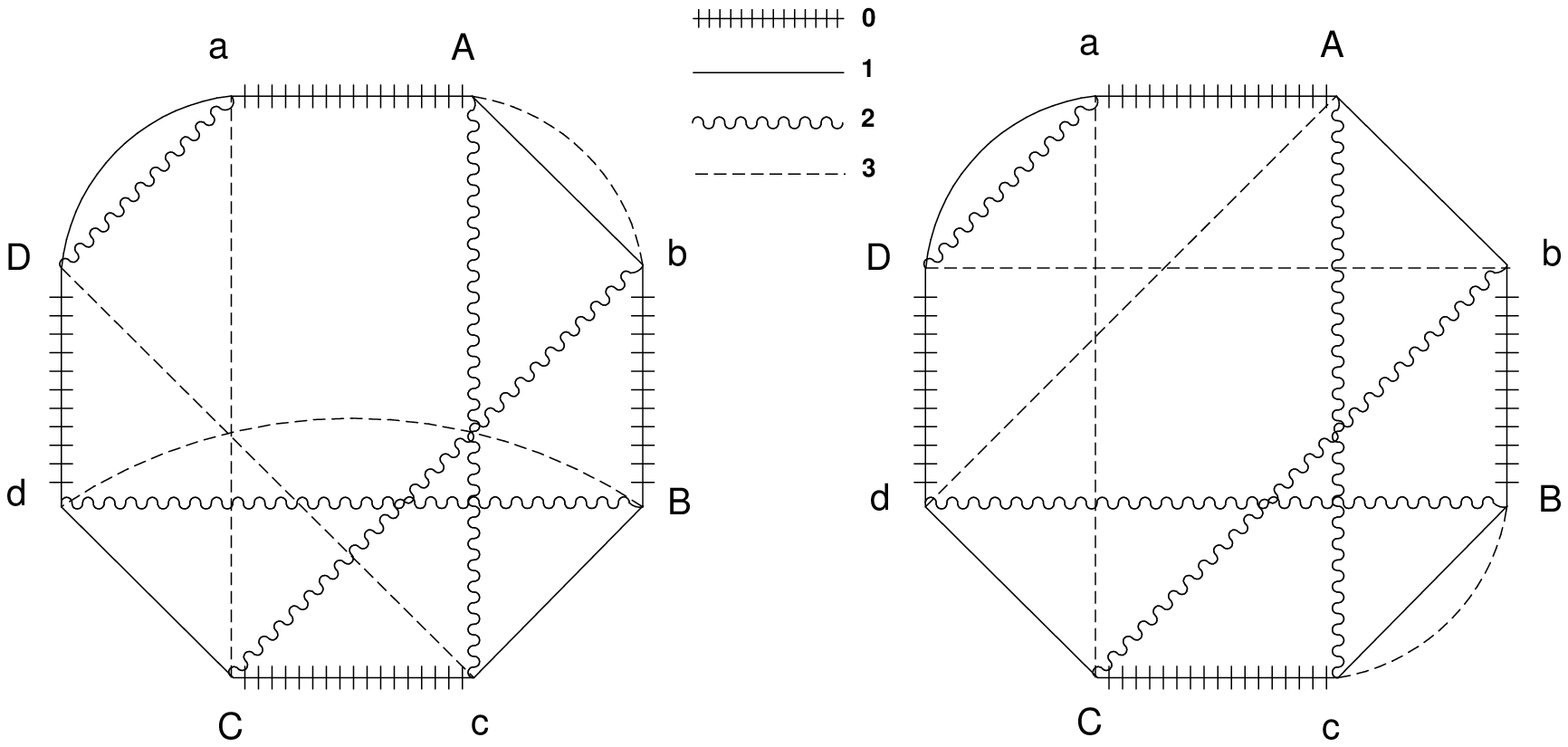}}}
\caption{Two non-isomorphic graphs representing the product of $\mathbf S^1$ with the 3-punctured sphere}\label{Fig11}
\end{figure}

As it is clear from Table \ref{table1}, the number of generated graphs quickly increases with the number of vertices and becomes very large even in the initial segment of the catalogues. However, Table \ref{table1} also shows that numbers are much smaller in the case of connected boundary, especially if it is toric. 

Therefore, as further development, we have restricted our attention to the study of orientable manifolds with (possibly disconnected) toric boundary. In a forthcoming paper (\cite{[CFMT]}), we have completed the classification of the manifolds involved in Proposition \ref{class_or}: all of them turn out to be complement of knots or links in $\mathbf S^3$. In particular, the two elements of $C^{(6)}$ represent the complements of the Hopf link and the trivial knot, while the manifolds $M_{\G_1},M_{\G_2},M_{\G_3}$ of the above proposition are complements of the links L8n7, L6n1 and L8n8 respectively (notations are according to Thistlethwaite Link Table).

\bigskip
\bigskip
\bigskip

{\it Acknowledgement.} Work performed under the auspices of the
G.N.S.A.G.A. of I.N.d.A.M. (Italy) and financially supported by
M.I.U.R. of Italy, University of Modena and Reggio Emilia and
University of Bologna, funds for selected research topics.
\smallskip

\noindent We would like to thank the referee for his helpful comments and suggestions.


\begin{thebibliography}{5}

\bibitem{[BCCGM]} Bandieri, P., Casali, M.R., Cristofori, P., Grasselli, L., Mulazzani, M.: Computational aspects of crystallization theory: complexity,
catalogues  and classification of $3$-manifolds. Atti Sem. Mat.
Fis. Univ. Modena {\bf 58}, 11-45 (2011)

\bibitem{[BCrG$_1$]} Bandieri, P., Cristofori, P., Gagliardi, C.: Nonorientable $3$-manifolds admitting coloured  triangulations with at most $30$ tetrahedra. J. Knot Theory Ramifications {\bf 18}, 381-395 (2009)

\bibitem{[BR]} Bandieri, P., Rivi, M.: Some bounds for the genus of $M^n\times I$. Note Mat.  {\bf 18}, 175-190 (1998)

\bibitem{[BM]} Bracho, J., Montejano, L.: The combinatorics of colored triangulations of manifolds. Geom. Dedicata {\bf 22} (3), 303-328 (1987) 

\bibitem{[C$_1$]} Casali, M.R: An equivalence criterion for $3$-manifolds. Rev. Mat. Univ. Complut. Madrid  {\bf 10}, 129-147 (1997)

\bibitem{[C$_2$]} Casali, M.R.: Computing Matveev's complexity of non-orientable
3-manifolds via crystallization theory. Topology Appl. {\bf 144}, 201-209 (2004)

\bibitem{[CC]} Casali, M.R., Cristofori, P.: A catalogue of orientable $3$-manifolds triangulated by $30$ coloured tethraedra. J. Knot Theory
Ramifications {\bf 17}, 579-599 (2008)

\bibitem{[CC$_1$]} Casali, M.R., Cristofori, P.: Computing Matveev's complexity via crystallization theory: the orientable case. Acta Appl. Math. {\bf 92}, 113-123 (2006)

\bibitem{[CC$_2$]} Casali, M.R., Cristofori, P.: Computing Matveev's complexity via crystallization theory: the boundary case. J. Knot Theory Ramifications {\bf 22 (8)}, 1350038 (30 pages) (2013). doi:10.1142/S0218216513500387

\bibitem{[CCM]} Casali, M.R., Cristofori, P., Mulazzani, M.: Complexity computation for compact 3-manifolds via crystallizations and Heegaard diagrams. Topology Appl. {\bf  159} (13), 3042-3048 (2012)

\bibitem{[CG]} Casali, M.R., Gagliardi, C.: A code for $m$-bipartite edge-coloured graphs. dedicated to the memory of Marco Reni, Rend. Ist. Mat. Univ. Trieste {\bf 32}, 55-76 (2001)

\bibitem{[CMV]} Cattabriga, A., Mulazzani, M., Vesnin, A.: Complexity, Heegaard diagrams and generalized Dunwoody manifolds. J. Korean Math. Soc. {\bf 47}, 585-599 (2010)

\bibitem{[Cr$_1$]} Cristofori, P.: Heegaard and regular genus agree for compact $3$-manifolds. Cahiers Topologie Geom. Differentielle Categ. {\bf 39}, 221-235 (1998)

\bibitem{[CFMT]} Cristofori, P., Fominykh, E., Mulazzani, M., Tarkaev, V.: 4-colored graphs and knot/link complements (2015, in preparation). 

\bibitem{[CGG]} Cristofori, P., Gagliardi, C., Grasselli, L.: Heegaard and
regular genus of $3$-manifolds with boundary.  Rev. Mat. Univ. Complut. Madrid  {\bf 8}, 379-398 (1995)

\bibitem{[E]} Epstein, D.B.A.: Factorization of 3-manifolds. Comment. Math. Helv. {\bf 36}, 91-102 (1961)

\bibitem{[FG1]} Ferri, M., Gagliardi, C.: A characterization of punctured $n$-spheres. Yokohama Math. J. {\bf 33}, 29-38 (1985)

\bibitem{[FGG]} Ferri, M., Gagliardi, C., Grasselli, L.: A
graph-theoretical representation of PL-manifolds. A survey on
crystallizations. Aequationes Math. {\bf 31}, 121-141 (1986)

\bibitem {[G$_0$]} Gagliardi, C.: Cobordant crystallizations. Discrete Math. {\bf 45}, 61-73 (1983)

\bibitem {[G$_1$]} Gagliardi, C.: On a class of $3$-dimensional polyhedra. Ann. Univ. Ferrara {\bf 33}, 51-88 (1987)

%\bibitem{[G$_2$]} Gagliardi, C.: Extending the concept of genus to dimension $n$. Proc. Amer. Math. Soc. {\bf 81}, 473-481 (1981)

\bibitem{[G$_3$]} Gagliardi, C.: Regular imbeddings of
edge--coloured graphs. Geom. Dedicata {\bf 11}, 297-314 (1981)

\bibitem {[G$_4$]} Gagliardi, C.: Regular genus: the boundary case. Geom. Dedicata {\bf 22}, 261-281 (1987)

\bibitem {[G$_5$]} Gagliardi, C.: How to deduce the fundamental group of a closed $n$-manifold from a contracted triangulation. J. Combin. Inf. Sist. Sci. {\bf 4}, 237-252 (1979)

\bibitem{[Gr]} Grasselli, L.: Edge-coloured graphs and associated groups. Rend. Circ. Mat. Palermo {\bf 12}, 263-269 (1986)

\bibitem{[He]} Hempel, J.: $3$-manifolds. Annals of Math. Studies, {\bf 86}, Princeton Univ. Press (1976)

\bibitem{[HeJ]} Hempel, J., Jaco, W.: 3-manifolds which fibers over a surface. Amer. J. Math. {\bf 94}, 189-205 (1972)

\bibitem{[HW]} Hilton, P.J., Wylie, S.: An introduction to
algebraic topology - Homology theory.  Cambridge Univ. Press,
(1960)

\bibitem{[Li]} Lins, S.: Gems, computers and attractors for
$3$-manifolds. Knots and Everything {\bf 5}, World Scientific
(1995)

\bibitem{[LM]} Lins, S., Mandel, A.: Graph-encoded 3-manifolds. Discrete Math. {\bf 57}, 261-284 (1985)

\bibitem{[M$_0$]} Matveev, S.: Transformations of special spines, and the Zeeman's conjecture.  Izv. Akad. Nuk SSSR Ser. Mat. {\bf 51}, 1104-1116 (1987)
(English trans. in: Math. USSR-Izv. {\bf 31}, 423-434 (1988))

\bibitem{[M$_1$]} Matveev, S.: Complexity theory of
three-dimensional manifolds. Acta Appl. Math. {\bf 19},
101-130 (1990)

\bibitem{[Ma]} Matveev, S.: Algorithmic topology and classification of
$3$-manifolds. ACM-Monographs {\bf 9}, Springer-Verlag,
Berlin-Heidelberg-New York (2003)

\bibitem{[P]} Pezzana, M.: Sulla struttura topologica delle variet\`a compatte. Atti Sem.
Mat. Fis. Univ. Modena {\bf 23}, 269-277 (1974)   

\bibitem{[Pi]} Piergallini, R.: Standard moves for standard polyhedra and spines. Rend. Circ. Mat. Palermo {\bf 18}, 391-414 (1988)

\bibitem{[GAP2013]} The GAP Group, GAP -- Groups, Algorithms, and Programming. Version 4.6.2 (2013). http://www.gap-system.org

\bibitem{[Wh]} White, A.T.: Graphs, groups and surfaces. North Holland (1973)

\end{thebibliography}
\end{document}